\documentclass[12pt]{amsart}       
\usepackage{txfonts}
\usepackage{amssymb}
\usepackage{eucal}
\usepackage{graphicx}
\usepackage{amsmath}
\usepackage{amscd}
\usepackage[all]{xy}           
\usepackage{amsfonts,latexsym}
\usepackage{xspace}
\usepackage{epsfig}
\usepackage{float}
\usepackage{color}
\usepackage{fancybox}
\usepackage{colordvi}
\usepackage{multicol}
\usepackage{colordvi}
\usepackage[colorlinks,final,backref=page,hyperindex,hypertex]{hyperref}
\usepackage[active]{srcltx} 

\usepackage{tikz}

\usepackage{graphicx}

\newtheorem{theorem}{Theorem}[section]
\newtheorem{prop}[theorem]{Proposition}
\newtheorem{lemma}[theorem]{Lemma}
\newtheorem{coro}[theorem]{Corollary}
\newtheorem{prop-def}{Proposition-Definition}[section]
\newtheorem{coro-def}{Corollary-Definition}[section]

\theoremstyle{definition}
\newtheorem{defn}[theorem]{Definition}
\newtheorem{remark}[theorem]{Remark}

\newcommand{\nc}{\newcommand}
\nc{\tred}[1]{\textcolor{red}{#1}}
\nc{\tblue}[1]{\textcolor{blue}{#1}}
\nc{\tgreen}[1]{\textcolor{green}{#1}}
\nc{\tpurple}[1]{\textcolor{purple}{#1}}
\nc{\btred}[1]{\textcolor{red}{\bf #1}}
\nc{\btblue}[1]{\textcolor{blue}{\bf #1}}
\nc{\btgreen}[1]{\textcolor{green}{\bf #1}}
\nc{\btpurple}[1]{\textcolor{purple}{\bf #1}}
\nc{\NN}{{\mathbb N}}
\nc{\ncsha}{{\mbox{\cyr X}^{\mathrm NC}}} \nc{\ncshao}{{\mbox{\cyr
X}^{\mathrm NC}_0}}

\newcommand{\efootnote}[1]{}

\renewcommand{\textbf}[1]{}

\newcommand{\delete}[1]{}

\nc{\mlabel}[1]{\label{#1}}  
\nc{\mcite}[1]{\cite{#1}}  
\nc{\mref}[1]{\ref{#1}}  
\nc{\mbibitem}[1]{\bibitem{#1}} 

\delete{
\nc{\mlabel}[1]{\label{#1}{\hfill \hspace{1cm}{\bf{{\ }\hfill(#1)}}}}
\nc{\mcite}[1]{\cite{#1}{{\bf{{\ }(#1)}}}}  
\nc{\mref}[1]{\ref{#1}{{\bf{{\ }(#1)}}}}  
\nc{\mbibitem}[1]{\bibitem[\bf #1]{#1}} 
}

\nc{\opa}{\ast} \nc{\opb}{\odot} \nc{\op}{\bullet} \nc{\pa}{\frakL}
\nc{\arr}{\rightarrow} \nc{\lu}[1]{(#1)} \nc{\mult}{\mrm{mult}}
\nc{\diff}{\mathfrak{Diff}}
\nc{\opc}{\sharp}\nc{\opd}{\natural}
\nc{\ope}{\circ}
\nc{\dpt}{\mathrm{d}}
\nc{\hck}{H_{RT}}
\nc{\vdf}{\calf}
\nc{\ldf}{\calf_\ell}
\nc{\hlf}{H_\ell}
\nc{\onek}{\mathbf{1}_\bfk}
\nc{\match}{matching\xspace}
\nc{\mrba}{matching Rota-Baxter algebra\xspace}
\nc{\Mrba}{Matching Rota-Baxter algebra\xspace}
\nc{\mrbas}{matching Rota-Baxter algebras\xspace}
\nc{\Mrbas}{Matching Rota-Baxter algebras\xspace}

\nc{\rba}{Rota-Baxter algebras\xspace}

\nc{\paybe}{polarized associative Yang-Baxter equation\xspace}
\nc{\Paybe}{Polarized associative Yang-Baxter equation\xspace}
\nc{\cpaybe}{PAYBE}

\nc{\diam}{alternating\xspace}
\nc{\Diam}{Alternating\xspace}
\nc{\cdiam}{canonical alternating\xspace}
\nc{\Cdiam}{Canonical alternating\xspace}
\nc{\AW}{\mathcal{A}}

\nc{\ari}{\mathrm{ar}}

\nc{\lef}{\mathrm{lef}}

\nc{\Sh}{\mathrm{ST}}

\nc{\Cr}{\mathrm{Cr}}

\nc{\st}{{Schr\"oder tree}\xspace}
\nc{\sts}{{Schr\"oder trees}\xspace}

\nc{\vertset}{\Omega} 

\nc{\assop}{\quad \begin{picture}(5,5)(0,0)
\line(-1,1){10}
\put(-2.2,-2.2){$\bullet$}
\line(0,-1){10}\line(1,1){10}
\end{picture} \quad \smallskip}

\nc{\operator}{\begin{picture}(5,5)(0,0)
\line(0,-1){6}
\put(-2.6,-1.8){$\bullet$}
\line(0,1){9}
\end{picture}}

\nc{\idx}{\begin{picture}(6,6)(-3,-3)
\put(0,0){\line(0,1){6}}
\put(0,0){\line(0,-1){6}}
\end{picture}}

\nc{\pb}{{\mathrm{pb}}}
\nc{\Lf}{{\mathrm{Lf}}}

\nc{\lft}{{left tree}\xspace}
\nc{\lfts}{{left trees}\xspace}

\nc{\fat}{{fundamental averaging tree}\xspace}

\nc{\fats}{{fundamental averaging trees}\xspace}
\nc{\avt}{\mathrm{Avt}}

\nc{\rass}{{\mathit{RAss}}}

\nc{\aass}{{\mathit{AAss}}}

\nc{\vin}{{\mathrm Vin}}    
\nc{\lin}{{\mathrm Lin}}    
\nc{\inv}{\mathrm{I}n}
\nc{\gensp}{V} 
\nc{\genbas}{\mathcal{V}} 
\nc{\bvp}{V_P}     
\nc{\gop}{{\,\omega\,}}     

\nc{\bin}[2]{ (_{\stackrel{\scs{#1}}{\scs{#2}}})}  
\nc{\binc}[2]{ \left (\!\! \begin{array}{c} \scs{#1}\\
    \scs{#2} \end{array}\!\! \right )}  
\nc{\bincc}[2]{  \left ( {\scs{#1} \atop
    \vspace{-1cm}\scs{#2}} \right )}  
\nc{\bs}{\bar{S}} \nc{\cosum}{\sqsubset} \nc{\la}{\longrightarrow}
\nc{\rar}{\rightarrow} \nc{\dar}{\downarrow} \nc{\dprod}{**}
\nc{\dap}[1]{\downarrow \rlap{$\scriptstyle{#1}$}}
\nc{\md}{\mathrm{dth}} \nc{\uap}[1]{\uparrow
\rlap{$\scriptstyle{#1}$}} \nc{\defeq}{\stackrel{\rm def}{=}}
\nc{\disp}[1]{\displaystyle{#1}} \nc{\dotcup}{\
\displaystyle{\bigcup^\bullet}\ } \nc{\gzeta}{\bar{\zeta}}
\nc{\hcm}{\ \hat{,}\ } \nc{\hts}{\hat{\otimes}}
\nc{\barot}{{\otimes}} \nc{\free}[1]{\bar{#1}}
\nc{\uni}[1]{\tilde{#1}} \nc{\hcirc}{\hat{\circ}} \nc{\lleft}{[}
\nc{\lright}{]} \nc{\lc}{\lfloor} \nc{\rc}{\rfloor}
\nc{\curlyl}{\left \{ \begin{array}{c} {} \\ {} \end{array}
    \right .  \!\!\!\!\!\!\!}
\nc{\curlyr}{ \!\!\!\!\!\!\!
    \left . \begin{array}{c} {} \\ {} \end{array}
    \right \} }
\nc{\longmid}{\left | \begin{array}{c} {} \\ {} \end{array}
    \right . \!\!\!\!\!\!\!}
\nc{\onetree}{\bullet} \nc{\ora}[1]{\stackrel{#1}{\rar}}
\nc{\ola}[1]{\stackrel{#1}{\la}}
\nc{\ot}{\otimes} \nc{\mot}{{{\boxtimes\,}}}
\nc{\otm}{\overline{\boxtimes}} \nc{\sprod}{\bullet}
\nc{\scs}[1]{\scriptstyle{#1}} \nc{\mrm}[1]{{\rm #1}}
\nc{\margin}[1]{\marginpar{\rm #1}}   
\nc{\dirlim}{\displaystyle{\lim_{\longrightarrow}}\,}
\nc{\invlim}{\displaystyle{\lim_{\longleftarrow}}\,}
\nc{\mvp}{\vspace{0.3cm}} \nc{\tk}{^{(k)}} \nc{\tp}{^\prime}
\nc{\ttp}{^{\prime\prime}} \nc{\svp}{\vspace{2cm}}
\nc{\vp}{\vspace{8cm}} \nc{\proofbegin}{\noindent{\bf Proof: }}
\nc{\proofend}{$\blacksquare$ \vspace{0.3cm}}
\nc{\modg}[1]{\!<\!\!{#1}\!\!>}
\nc{\intg}[1]{F_C(#1)} \nc{\lmodg}{\!
<\!\!} \nc{\rmodg}{\!\!>\!}
\nc{\cpi}{\widehat{\Pi}}
\nc{\sha}{{\mbox{\cyr X}}}  

\newfont{\scyr}{wncyr10 scaled 550}
\nc{\ssha}{\mbox{\bf \scyr X}}
\nc{\shap}{{\mbox{\cyrs X}}} 
\nc{\shpr}{\diamond}    
\nc{\shp}{\ast} \nc{\shplus}{\shpr^+}
\nc{\shprc}{\shpr_c}    
\nc{\msh}{\ast} \nc{\zprod}{m_0} \nc{\oprod}{m_1}
\nc{\vep}{\epsilon} \nc{\labs}{\mid\!} \nc{\rabs}{\!\mid}
\nc{\sqmon}[1]{\langle #1\rangle}

\nc{\mmbox}[1]{\mbox{\ #1\ }} \nc{\dep}{\mrm{dep}} \nc{\fp}{\mrm{FP}}
\nc{\rchar}{\mrm{char}} \nc{\End}{\mrm{End}} \nc{\Fil}{\mrm{Fil}}
\nc{\Mor}{Mor\xspace} \nc{\gmzvs}{gMZV\xspace}
\nc{\gmzv}{gMZV\xspace} \nc{\mzv}{MZV\xspace}
\nc{\mzvs}{MZVs\xspace} \nc{\Hom}{\mrm{Hom}} \nc{\id}{\mrm{id}}
\nc{\im}{\mrm{im}} \nc{\incl}{\mrm{incl}} \nc{\map}{\mrm{Map}}
\nc{\mchar}{\rm char} \nc{\nz}{\rm NZ} \nc{\supp}{\mathrm Supp}

\nc{\Alg}{\mathbf{Alg}} \nc{\Bax}{\mathbf{Bax}} \nc{\bff}{\mathbf f}
\nc{\bfk}{{\bf k}} \nc{\bfone}{{\bf 1}} \nc{\bfx}{\mathbf x}
\nc{\bfy}{\mathbf y}
\nc{\base}[1]{\bfone^{\otimes ({#1}+1)}} 
\nc{\Cat}{\mathbf{Cat}}

\nc{\detail}{\marginpar{\bf More detail}
    \noindent{\bf Need more detail!}
    \svp}
\nc{\Int}{\mathbf{Int}} \nc{\Mon}{\mathbf{Mon}}
\nc{\rbtm}{{shuffle }} \nc{\rbto}{{Rota-Baxter }}
\nc{\remarks}{\noindent{\bf Remarks: }} \nc{\Rings}{\mathbf{Rings}}
\nc{\Sets}{\mathbf{Sets}} \nc{\wtot}{\widetilde{\odot}}
\nc{\wast}{\widetilde{\ast}} \nc{\bodot}{\bar{\odot}}
\nc{\bast}{\bar{\ast}} \nc{\hodot}[1]{\odot^{#1}}
\nc{\hast}[1]{\ast^{#1}} \nc{\mal}{\mathcal{O}}
\nc{\tet}{\tilde{\ast}} \nc{\teot}{\tilde{\odot}}
\nc{\oex}{\overline{x}} \nc{\oey}{\overline{y}}
\nc{\oez}{\overline{z}} \nc{\oef}{\overline{f}}
\nc{\oea}{\overline{a}} \nc{\oeb}{\overline{b}}
\nc{\weast}[1]{\widetilde{\ast}^{#1}}
\nc{\weodot}[1]{\widetilde{\odot}^{#1}} \nc{\hstar}[1]{\star^{#1}}
\nc{\lae}{\langle} \nc{\rae}{\rangle}
\nc{\lf}{\lfloor}
\nc{\rf}{\rfloor}

\nc{\QQ}{{\mathbb Q}}
\nc{\RR}{{\mathbb R}} \nc{\ZZ}{{\mathbb Z}}

\nc{\cala}{{\mathcal A}} \nc{\calb}{{\mathcal B}}
\nc{\calc}{{\mathcal C}}
\nc{\cald}{{\mathcal D}} \nc{\cale}{{\mathcal E}}
\nc{\calf}{{\mathcal F}} \nc{\calg}{{\mathcal G}}
\nc{\calh}{{\mathcal H}} \nc{\cali}{{\mathcal I}}
\nc{\call}{{\mathcal L}} \nc{\calm}{{\mathcal M}}
\nc{\caln}{{\mathcal N}} \nc{\calo}{{\mathcal O}}
\nc{\calp}{{\mathcal P}} \nc{\calr}{{\mathcal R}}
\nc{\cals}{{\mathcal S}} \nc{\calt}{{\mathcal T}}
\nc{\calu}{{\mathcal U}} \nc{\calw}{{\mathcal W}} \nc{\calk}{{\mathcal K}}
\nc{\calx}{{\mathcal X}} \nc{\CA}{\mathcal{A}}

\nc{\fraka}{{\mathfrak a}} \nc{\frakA}{{\mathfrak A}}
\nc{\frakb}{{\mathfrak b}} \nc{\frakB}{{\mathfrak B}}
\nc{\frakc}{{\mathfrak c}}
\nc{\frakD}{{\mathfrak D}} \nc{\frakF}{\mathfrak{F}}
\nc{\frakf}{{\mathfrak f}} \nc{\frakg}{{\mathfrak g}}
\nc{\frakH}{{\mathfrak H}} \nc{\frakL}{{\mathfrak L}}
\nc{\frakM}{{\mathfrak M}} \nc{\bfrakM}{\overline{\frakM}}
\nc{\frakm}{{\mathfrak m}} \nc{\frakP}{{\mathfrak P}}
\nc{\frakN}{{\mathfrak N}} \nc{\frakp}{{\mathfrak p}}
\nc{\frakS}{{\mathfrak S}} \nc{\frakT}{\mathfrak{T}}
\nc{\frakX}{{\mathfrak X}}
\nc{\BS}{\mathbb{S
}}

\font\cyr=wncyr10 \font\cyrs=wncyr7
\nc{\li}[1]{\textcolor{red}{#1}}
\nc{\lir}[1]{\textcolor{red}{Li:#1}}
\nc{\yi}[1]{\textcolor{blue}{Yi: #1}}
\nc{\xing}[1]{\textcolor{purple}{Xing:#1}}
\nc{\revise}[1]{\textcolor{red}{#1}}
\nc{\song}[1]{\textcolor{green}{Song:#1}}
\nc{\chia}[1]{\textcolor{orange}{Chia:#1}}

\nc{\ID}{{\rm I}}\nc{\lbar}[1]{\overline{#1}}\nc{\bre}{{\rm bre}}
\nc{\sd}{\cals}\nc{\rb}{\rm RB}\nc{\A}{\rm A}\nc{\LL}{\rm L}\nc{\tx}{\tilde{X}}
\nc{\col}{\Delta_{RT}}\nc{\mul}{m_{RT}}\nc{\ul}{u_{RT}}\nc{\epl}{\varepsilon_{RT}}
\nc{\hl}{H_{RT}}\nc{\arro}[1]{#1}\nc{\px}{P_{\tx}}\nc{\pw}{P_{\mathfrak{w}}}\nc{\pl}{B_\omega^+}
\nc{\pp}{\pl}\nc{\ppp}[1]{B^+(#1)}\nc{\dw}{\diamond_{\mathfrak{w}}}\nc{\dl}{\diamond_{\rm \ell}}
\nc{\ncshaw}{\sha^{{\rm NC}}_{\Omega}}\nc{\ncshal}{\sha^{{\rm NC}}_{{\rm RT}}}
\nc{\ver}{\rm V}\nc{\ld}{l}\nc{\del}{\Delta_{{\rm \ell}}}\nc{\epsl}{\epsilon_{{\rm \ell}}}
\nc{\uul}{u_{{\rm \ell}}}\nc{\oneh}{\mathbf{1}}\nc{\onew}{\mathbf{1}}
\nc{\etree}{1} \nc{\conc}{m_{RT}}
\nc{\hrtb}{H_{RT}(X\sqcup\Omega)} \nc{\hrts}{H_{RT}(X, \Omega)}\nc{\rts}{\mathcal{T}(X, \Omega)}\nc{\rfs}{\mathcal{F}(X, \Omega)} \nc{\ncshall}{\sha^{{\rm NC}}_{{\rm RT}}} \nc{\ldl}{\leq_{\mathrm{dl}}} \nc{\pla}{B_{\alpha}^{+}} \nc{\plb}{B_{\beta}^{+}}

\nc{\bim}[1]{#1}  \nc{\shaop}{\sha_{\Omega}^{+}}  \nc{\shao}{\sha_{\Omega}}
\nc{\bbim}[2]{#1 #2} \nc{\bbbim}[2]{#1,\, #2} \nc{\RBF}{{\rm RBF}}
\nc{\frbf}{F_{\RBF}} \nc{\shaf}{\ssha_{\tiny{\Omega}}} \nc{\sham}{\diamond_{\Omega}}

\nc{\dnx}{\Delta_n A} \nc{\dx}{\Delta A} \nc{\dgp}{{\rm deg_{P}}}
\nc{\dgt}{{\rm deg_{T}}} \nc{\dg}{{\rm deg}} \nc{\ida}{ID($A$)} \nc{\tu}{\tilde{u}} \nc{\tv}{\tilde{v}}
 \nc{\fbase}{\calb} \nc{\LF}{\mathrm{RF}} \nc{\FFA}{\mathrm{LF}} \nc{\irr}{\mathrm{Irr}}
 \nc{\result}{\bfk\mathrm{Irr}(S_n)}  \nc{\I}{I_{\mathrm{ID},n}^0}
 \nc{\nrs}{\calr_n^\star} \nc{\ii}{\mathrm{I}} \nc{\iii}{\mathrm{II}}
\nc{\intl}{{\rm int}}\nc{\ws}[1]{{#1}}\nc{\deleted}[1]{\delete{#1}}\nc{\plas}{placements\xspace}

\nc{\Id}{\mathrm{Id}} \nc{\Irr}{\mathrm{Irr}}

\nc{\tos}{totally ordered set } \nc{\nes}{nonempty set}

\begin{document}

\title[Matching Hom-algebraic structures]{Matching Hom-Setting of Rota-Baxter algebras, dendriform algebras and pre-Lie algebras}
%
\author{Dan Chen}
\address{School of Mathematics and Statistics, Lanzhou University, Lanzhou, Gansu 730000, P.\,R. China}
\email{gaoxing@lzu.edu.cn}

\author{Xiao-Song Peng}
\address{School of Mathematics and Statistics,
Lanzhou University, Lanzhou, Gansu 730000, P.\,R. China}
\email{pengxiaosong3@163.com}

\author{Chia Zargeh}
\address{Instituto de Matem\'atica e Estat\'istica, Universidade de S\~ao Paulo, S\~ao Paulo, SP, Brasil}\email{ch.zarge@gmail.com}

\author{Yi Zhang}
\address{School of Mathematics and Statistics, Lanzhou University, Lanzhou, Gansu 730000, P.\,R. China}
\email{zhangy2016@lzu.edu.cn}


\date{\today}
\begin{abstract}
In this paper, we introduce the Hom-algebra setting of the notions of matching Rota-Baxter algebras, matching (tri)dendriform algebras and matching (pre)Lie algebras. Moreover, we study the properties and relationships between categories of these matching Hom-algebraic structures.
\end{abstract}

\subjclass[2010]{
16W99, 
17B38, 
17B61. 
}

\keywords{Matching Hom-Rota-Baxter algebras; Matching Hom-dendriform algebras; Matching Hom-pre-Lie algebras; Matching Hom-Lie algebras}

\maketitle

\tableofcontents

\setcounter{section}{0}

\allowdisplaybreaks

\section{Introduction}

\subsection{Hom-algebraic structures}
The origin of Hom-structures may be found in the study of Hom-Lie algebras which were first introduced by Hartwig, Larsson and Silvestrov~\mcite{HLS06}. Hom-Lie algebras, as a generalization of Lie algebras, are introduced to describe the structures on  deformations of the
Witt algebra and the Virasoro algebra. More precisely, a Hom-Lie algebra is a triple $(L, [-,-], \alpha)$ consisting of a {\bf k}-module $L$, a bilinear skew-symmetric bracket $[-, -]: L\ot L \rightarrow L$ and an algebra endomorphism $\alpha: L \rightarrow L$ satisfying the following Hom-Jacobi identity:
\begin{align*}
[\alpha(x), [y,z]]+[\alpha(y), [z, x]]+[\alpha(x), [x,y]]=0 \, \text{ for all } \, x, y,z \in L.
\end{align*}

Recently, there have  been several interesting developments of Hom-Lie algebras in mathematics and mathematical physics, including  Hom-Lie bialgebras~\mcite{CWZ12, CZZ20}, quadratic Hom-Lie algebras~\mcite{BM14}, involutive Hom-semigroups~\mcite{ZG14}, deformed vector fields and differential calculus~\mcite{LS051}, representations~\mcite{S12, ZCM16}, cohomology and homology theory~\mcite{AEM11, Y09}, Yetter-Drinfeld categories~\mcite{WW14}, Hom-Yang-Baxter equations~\mcite{CZ15, CZ18, LMP, SB14, Y092}, Hom-Lie 2-algebras~\mcite{SC13, ST18}, $(m,n)$-Hom-Lie algebras~\mcite{MaZ18}, Hom-left-symmetric algebras~\mcite{MS08} and enveloping algebras~\mcite{GZZ18}. In particular, the Hom-Lie algebra on semisimple Lie
algebras was studied in~\mcite{JL08} and the Hom-Lie structure on affine Kac-Moody was constructed in~\mcite{MZ18}.

In 2008, Makhlouf and Silvestrov~\mcite{MS08} introduced the notation of  Hom-associative algebras whose associativity law is twisted by a linear map.
Usual functors between the categories of Lie algebras and associative algebras have been extended to the Hom-setting. It is shown that a Hom-associative algebra gives rise to a Hom-Lie algebra using the commutator. Since then, various  Hom-analogues of some classical algebraic structures  have been introduced and studied intensively, such as Hom-coalgebras, Hom-bialgebras and Hom-Hopf algebras~\mcite{MS081, MS10}, Hom-groups~\mcite{H19,H191}, Hom-Hopf modules~\mcite{GZW15},  Hom-Lie superalgebras~\mcite{ABCS18, MCZ18}, generalize Hom-Lie algebras~\mcite{MDL16}, Hom-Poisson algebras~\mcite{AC16}.

Dendriform algebras were introduced by Loday \mcite{L01} with motivation from algebraic $K$-theory.  Latter, tridendriform algebras  were proposed by Loday and Ronco \mcite{LR03} in their study of polytopes and Koszul duality. The classical links between Rota-Baxter algebras and (tri)dendriform algebras were given in \mcite{A00, E02}, resembling the structure of Lie algebras on an associative algebra. In 2012, Makhlouf~\mcite{M12} generalized the concepts of  dendriform algebras and Rota-Baxter algebras by twisting the identities by mean of a linear map, which
 were called  Hom-dendriform algebras and Rota-Baxter Hom-algebras, respectively.  The connections between all these categories of Hom algebras were also investigated in~\mcite{M12}. Due to the fundamental work of  Makhlouf~\mcite{M12}, we have the following commutative diagram of categories (the arrows will go in the opposite direction for the corresponding operads)
\begin{equation*}
\begin{split}
\xymatrix@C2.0em{ \text{Rota-Baxter} \atop \text{Hom-associative algebra} \ar[rr]^{\text{commutator}} \ar[d] && \text{Rota-Baxter} \atop \text{Hom-Lie algebra} \ar[d] \\
\scriptsize{\text{Hom-dendriform algebra}} \ar[rr]^{\text{commutator}} && \scriptsize{\text{Hom-pre-Lie algebra}}
}
\end{split}
\mlabel{di:old}
\end{equation*}
\subsection{Motivations for matching Hom-algebraic structures}

The recent concept of a matching or multiple Rota-Baxter~\cite{GGZ19} came from the study of multiple pre-Lie algebras~\mcite{Foi18} originated from the pioneering work of Bruned, Hairer and Zambotti~\mcite{BHZ} on algebraic renormalization of regularity structures. It is shown that the \mrba was motivated by the studies of associative Yang-Baxter equations, Volterra integral equations and linear structure of Rota-Baxter operators~\cite{GGZ19}.
More precisely, for exploring the relationship between  associative Yang-Baxter equations and  classical Yang-Baxter equations, Aguiar~\mcite{A01} proposed a polarized form of the expression on the left hand side of the associative Yang-Baxter equation:
\begin{align*}
\{r,s\} :=r_{13}s_{12}-r_{12}s_{23}+r_{23}s_{13},
\end{align*}
where $r, s\in A\ot A$ and $A$ is a unitary associative algebra. The corresponding equation
\begin{equation*}
r_{13}s_{12}-r_{12}s_{23}+r_{23}s_{13} = 0
\end{equation*}
was called  polarized associative Yang-Baxter equation (PAYBE) by Guo and etc~\cite{GGZ19}. Paralleled to the fact that solutions of the associative Yang-Baxter equation naturally give Rota-Baxter operators,
the matching Rota-Baxter operators are determined by solutions of a PAYBE~\cite{GGZ19}.

The basic theory of \mrbas  was originally established  in~\mcite{GGZ19, GGZy}, has proven useful not only in (compatible) multiple operations~\mcite{LBS, ZBG1, ZBG2, ZyyG19, ZGM20}, but also in other areas of mathematics as well, such as polarized associative Yang-Baxter equation~\mcite{GGZ19}, algebraic combinatorics~\mcite{GGZ19, GGZ20}, matching shuffle product~\mcite{GGZy}, algebraic integral equation~\mcite{GGL}, Gr\"obner-Shirshov bases and Hopf algebras~\mcite{GGZ20}.
Based on the close relationships between \match Rota-Baxter algebras , \match dendriform algebras and \match pre-Lie algebras, Guo et al.~\cite{GGZ19}  previously showed the following commutative diagram of categories.
\begin{equation*}
\begin{split}
\xymatrix@C2.4em{ \text{\match Rota-Baxter} \atop \text{associative algebra} \ar[rr]^{\text{commutator}} \ar[d] &&  \text{\match Rota-Baxter} \atop \text{Lie algebra} \ar[d] \\
\text{\match}\atop \text{dendriform algebra} \ar[rr]^{\text{commutator}} && \text{\match (multiple)}\atop \text{pre-Lie algebra}
}
\end{split}
\end{equation*}

The main purpose of this paper is to extend these matching algebraic structures to the Hom-algebra setting and study the connections between these categories of Hom-algebras.
These results give rise to the following commutative diagram of categories.
\begin{equation*}
\begin{split}
\xymatrix@C2.4em{ \text{matching Hom-associative} \atop \text{Rota-Baxter algebra} \ar[rr]^{\text{commutator}} \ar[d] &&  \text{matching Hom-Lie} \atop \text{Rota-Baxter algebra} \ar[d] \\
\text{matching}\atop \text{Hom-dendriform algebra} \ar[rr]^{\text{commutator}} && \text{matching (multiple)}\atop \text{Hom-pre-Lie algebra}
}
\end{split}
\end{equation*}
We would like to emphasize that the notation of matching Hom-Lie Rota-Baxter algebras will play a curial role in mathematical physics. The Rota-Baxter equation on a Lie algebra is the operator form of the classical Yang-Baxter equation~\mcite{S83}. Similarly, there should be a close relationship between the \match Hom Rota-Baxter equation in~(\mref{eq:lrbeq}) with weight zero and the polarized classical Yang-Baxter equation, as a Hom-Lie algebra variation of the Hom version of \paybe.

\subsection{Outline of the paper and summary of results}
In section~\mref{sec1}, we provide definitions concerning generalization of matching associative algebras, matching pre-Lie algebras to Hom-algebras setting and describe some specific cases of matching Hom-algebraic structures. Also, the close relationship between matching Hom-Lie algebras and Hom-Lie algebras will be shown.

In section~\mref{sec2}, we extend the notion of matching Rota-Baxter algebras to the Hom-associative algebra setting. It is also shown that matching Hom-associative Rota-Baxter algebras  can be reduced from a matching Rota-Baxter algebra. At the end of this section, the construction of Hom-algebras using elements of the centroid is generalized to the matching Rota-Baxter algebras.

Section~\mref{sec3} is devoted to the definition of matching Hom-(tri)dendriform algebras and the approach of construction of a matching Hom-(tri)dendiform algebra from a matching (tri)dendiform algebra. Some results related to the connections between matching Hom-(tri)dendiform algebras and compatible Hom-associative algebras as well as between matching Hom-dendriform algebras and matching Hom-preLie algebras will be established.

In section~\mref{sec4}, the concepts of matching Hom-Lie Rota-Baxter algebras and matching Rota-Baxter algebras  involving elements of the centroid of matching Lie Rota-Baxter algebras will be established. Also, some results related to the connection between matching Hom-Lie Rota-Baxter algebra of weight zero and matching Hom-preLie algebra will be obtained.

\smallskip

\noindent
{\bf Notation.}
Throughout this paper, let $\bfk$ be a unitary commutative ring unless the contrary is specified,
which will be the base ring of all modules, algebras,  tensor products, operations as well as linear maps.
We always suppose that $\Omega$ is a nonempty set. We denote by ${P}_\Omega:=({P}_\omega)_{\omega\in \Omega}$ the collection of operations ${P}_\omega$, $\omega\in \Omega$, where $\Omega$ is  a set indexing the linear operators.

\section{Matching Hom-associative, matching Hom-preLie and matching Hom-Lie algebras} \mlabel{sec1}
In this section, we give the definitions of matching Hom-associative algebras, compatible Hom-associative algebras, compatible Hom-preLie algebras and compatible Hom-Lie algebras, which generalize the corresponding matching algebraic structures introduced in~\mcite{GGZ19}. Then we explore the relationships between these categories from the point of view of  Hom-algebras.

\begin{defn}
A {\bf matching Hom-associative algebra} is a \bfk-module $A$ together with a collection of binary operations $\cdot_{\omega}: A \ot A \rightarrow A, \omega \in \Omega$ and a linear map $p: A \rightarrow A$ such that
\begin{align*}
(x \cdot_{\alpha} y) \cdot_{\beta} p(z)=p(x) \cdot_{\alpha} (y \cdot_{\beta} z) \, \text{ for all }\, x,y,z \in A \text{ and } \alpha, \beta \in \Omega.
\end{align*}
A matching Hom-associative algebra is called {\bf totally compatible} if it satisfies
\begin{align*}
(x \cdot_{\alpha} y) \cdot_{\beta} p(z)=p(x) \cdot_{\beta} (y \cdot_{\alpha} z) \, \text{ for all }\, x,y,z \in A \text{ and } \alpha, \beta \in \Omega.
\end{align*}
\end{defn}

More generally,
\begin{defn}
A {\bf compatible Hom-associative algebra} is a \bfk-module $A$ together with a collection of binary operations $\cdot_{\omega}:A \ot A \rightarrow A, \omega \in \Omega$ and a linear map $p: A \rightarrow A$ such that
\begin{align}\mlabel{eq:chaa}
(x \cdot_{\alpha} y) \cdot_{\beta} p(z) + (x \cdot_{\beta} y) \cdot_{\alpha} p(z)=p(x) \cdot_{\alpha}(y \cdot_{\beta} z)+ p(x) \cdot_{\beta} (y \cdot_{\alpha} z)
\end{align}
for all $x,y,z \in A$ and $\alpha, \beta \in \Omega$. For simplicity, we denote it by $(A, \cdot_\Omega, p)$.
\end{defn}

\begin{remark}
\begin{enumerate}
  \item Any matching Hom-associative algebra or totally compatible Hom-associative algebra is a compatible Hom-associative algebra.
  \item By taking $p=\id$, we recover to the definition of matching associative algebras, totally compatible associative algebra and compatible associative algebra given in~\mcite{GGZ19}.
  \item If $\Omega$ is a singleton and the characteristic of {\bf k} is not 2, then the notation of matching Hom-associative algebras and the notation of compatible Hom-associative algebras are equivalent and recover to the Hom-associative algebras introduced in~\mcite{MS08}.
\end{enumerate}
\end{remark}

\begin{defn}
A {\bf matching Hom-Lie algebra} is a \bfk-module $\mathfrak{g}$ equipped with a collection of binary operations $[,]_{\omega}: \mathfrak{g} \ot \mathfrak{g} \rightarrow \mathfrak{g}$, $\omega \in \Omega$ and a linear map $p: \mathfrak{g} \rightarrow \mathfrak{g}$ such that
\begin{align*}
[x,x]_{\omega}=0
\end{align*}
and
\begin{align}
[p(x), [y,z]_{\beta}]_{\alpha}+ [p(y),[z,x]_{\alpha}]_{\beta}+[p(z),[x,y]_{\alpha}]_{\beta}=0
\mlabel{eq:defmhla}
\end{align}
for all $x,y,z \in \mathfrak{g}$ and $\alpha,\beta, \omega \in \Omega$.
\end{defn}

\begin{remark}
A totally compatible Hom-associative algebra $(A,\cdot_\Omega,p)$ has a natural matching Hom-Lie algebra structure with the Lie bracket defined by
\begin{align*}
[x,y]_{\omega} :=x \cdot_{\omega}y-y \cdot_{\omega} x, \,\, \text{for $x,y \in A$ and $\omega \in \Omega$.}
\end{align*}
\end{remark}
The \match Hom-Lie algebra has a close relationship with Hom-Lie algebras. We first record a lemma for a preparation.

\begin{lemma}\label{lemm:mmlie}
 Let $(\mathfrak{g},[,]_{\Omega}, p)$ be a \match Hom-Lie algebra. Consider linear combinations
 \begin{align}
[,]_{A}:=\sum_{\alpha\in \Omega} a_\alpha [,]_\alpha \,\, \text{ and } \,\, [,]_{B}:=\sum_{\beta\in \Omega} b_\beta [,]_\beta,
\mlabel{eq:mmlieid1}
\end{align}
where $a_\alpha, b_\beta \in \bfk$ for $\alpha, \beta \in \Omega$ with  finite supports. Then
 \begin{align*}
 [p(x),[y,z]_{B}]_{A}+[p(y),[z,x]_{A}]_{B}+[p(z),[x,y]_{A}]_{B}=0 \,\, \text{ for }\, x, y,z \in \mathfrak{g}.
 \end{align*}
\end{lemma}

\begin{proof}
By Eq.~(\mref{eq:mmlieid1}), for $x, y,z \in \mathfrak{g}$, we have
\begin{align*}
 [p(x),[y,z]_{B}]_{A}
=&\ [p(x),\sum_{\beta\in \Omega} b_\beta [y,z]_\beta]_{A}
=\sum_{\alpha\in \Omega} a_\alpha [p(x),\sum_{\beta\in \Omega} b_\beta [y,z]_\beta]_\alpha\\
=&\ \sum_{\alpha\in \Omega}\sum_{\beta\in \Omega} a_\alpha b_\beta [p(x), [y,z]_\beta]_\alpha.
\end{align*}
Similarly, we also have
\begin{align*}
[p(y),[z,x]_{A}]_{B}=\sum_{\alpha\in \Omega}\sum_{\beta\in \Omega} b_\beta a_\alpha [p(y), [z,x]_\alpha]_\beta\, \text{ and }\,
[p(z),[x,y]_{A}]_{B}= \sum_{\alpha\in \Omega}\sum_{\beta\in \Omega} b_\beta a_\alpha [p(z), [x,y]_\alpha]_\beta.
\end{align*}
Since $(\mathfrak{g},[,]_{\Omega}, p\})$ is a \match Hom-Lie algebra, then
\begin{align*}
[p(x),[y,z]_\beta]_\alpha+[p(y),[z,x]_\alpha]_\beta+[p(z),[x,y]_\alpha]_\beta=0\, \text{ for all }\, x, y, z \in \mathfrak{g} \,\, \text{and}\,\,   \alpha, \beta \in \Omega.
\end{align*}
Thus
\begin{align*}
 [p(x),[y,z]_{B}]_{A}+[p(y),[z,x]_{A}]_{B} +[p(z),[x,y]_{A}]_{B}=0,
 \end{align*}
 as desired.
\end{proof}

\begin{prop}
Let $(\mathfrak{g}, [,]_{\Omega}, p)$ be a \match Hom-Lie algebra. Consider linear combinations
\begin{align}
[,]_A:=\sum_{\omega\in \Omega} a_\omega [,]_\omega,\,  a_\omega \in \bfk,
\mlabel{eq:mmlieid2}
\end{align}
with a finite support. Then $(\mathfrak{g}, [,]_A)$ is a Hom-Lie algebra.
\end{prop}
\begin{proof}
It follows from Lemma~\mref{lemm:mmlie}  by taking $(a_\omega)_{\omega\in \Omega}=(b_\omega)_{\omega\in \Omega}$.
\end{proof}

More generally, we propose
\begin{defn}
A {\bf compatible Hom-Lie algebra} is a \bfk-module $\mathfrak{g}$ together with a set of binary operations $[,]_{\omega}: \mathfrak{g} \ot \mathfrak{g} \rightarrow \mathfrak{g}, \omega \in \Omega $ and a linear map $p: \mathfrak{g} \rightarrow \mathfrak{g}$ such that
\begin{align*}
[x,x]_{\omega}=0
\end{align*}
and
\begin{align}
[p(x),[y,z]_{\alpha}]_{\beta}&\ + [p(y),[z,x]_{\alpha}]_{\beta}+[p(z),[x,y]_{\alpha}]_{\beta} \notag\\
&\ +[p(x),[y,z]_{\beta}]_{\alpha}+[p(y),[z,x]_{\beta}]_{\alpha}+[p(z),[x,y]_{\beta}]_{\alpha}=0 \mlabel{eq:HomJaco}
\end{align}
for all $x,y,z \in \mathfrak{g}$ and $\omega, \alpha, \beta \in \Omega$.
\end{defn}

\begin{remark}
\begin{enumerate}
\item Every  matching  Hom-Lie algebra is a compatible Hom-Lie algebra.
\item Given two Hom-Lie algebras $(\mathfrak{g}, [,]_\alpha,p)$  and $(\mathfrak{g}, [,]_\beta,p)$. Define a new bracket $[,]: \mathfrak{g} \ot \mathfrak{g} \rightarrow \mathfrak{g}$ as follows:
    \begin{align*}
    [x, y]:=a_\alpha[x, y]_\alpha+b_\beta[x, y]_\beta\, \text{ for some }\, a_\alpha, b_{\beta} \in \bfk.
    \end{align*}
   Clearly, this new bracket is both skew symmetric and bilinear. Then $(\mathfrak{g}, [,],p)$ is further a Hom-Lie algebra if $[, ]$ satisfies the Hom-Jacobi identity
   $$[p(x),[y,z]]+[p(y),[z,x]]+[p(z),[x,y]]=0.$$  By a direct calculation,  we get that this condition is equivalent to Eq.~(\mref{eq:HomJaco}).
\end{enumerate}
\end{remark}

\begin{prop}
Let $(\mathfrak{g}, [,]_{\Omega}, p)$ be a matching Hom-Lie algebra. Then for $x,y,z \in \mathfrak{g}$ and $\alpha, \beta \in \Omega$, we have
\begin{align*}
&\ [p(x),[y,z]_{\alpha}]_{\beta}= [p(x), [y,z]_{\beta}]_{\alpha}, \\
&\ [p(x), [y,z]_{\alpha}]_{\beta}+[p(y),[z,x]_{\alpha}]_{\beta}+[p(z),[x,y]_{\alpha}]_{\beta}=0.
\end{align*}
\end{prop}

\begin{proof}
Since Eq.~(\mref{eq:defmhla}) holds for any $x,y,z \in A$ and $\alpha, \beta \in \Omega$, we get
\begin{align}
[p(y), [z,x]_{\alpha}]_{\beta}+[p(z),[x,y]_{\beta}]_{\alpha}+[p(x),[y,z]_{\beta}]_{\alpha}=0.
\mlabel{eq:tran}
\end{align}
Eqs.~(\mref{eq:defmhla}) and (\mref{eq:tran}) result in
\begin{align*}
[p(z), [x,y]_{\alpha}]_{\beta}-[p(z),[x,y]_{\beta}]_{\alpha}=0.
\end{align*}
By the arbitrariness of $x,y,z$, we have
\begin{align*}
[p(x),[y,z]_{\alpha}]_{\beta}= [p(x), [y,z]_{\beta}]_{\alpha}
\end{align*}
and so
\begin{align*}
[p(x), [y,z]_{\alpha}]_{\beta}+[p(y),[z,x]_{\alpha}]_{\beta}+[p(z),[x,y]_{\alpha}]_{\beta}=0.
\end{align*}
\end{proof}

Generalizing the well known result that an associative algebra has a Lie algebra structure via the commutator bracket, we show that a compatible Hom-associative algebra has a compatible Hom-Lie algebra structure.
\begin{prop}
Let $(A, \cdot_\Omega, p)$ be a compatible Hom-associative algebra. Then $(A,[,]_{\Omega}, p)$ is a compatible Hom-Lie algebra, where
\begin{align}
[,]_{\omega}: A \ot A \rightarrow A, [x,y]_{\omega}:=x \cdot_{\omega}y -y \cdot_{\omega} x \,\, \text{ for } x,y \in A \text{ and } \omega \in \Omega.
\mlabel{eq:def[]}
\end{align}
\end{prop}

\begin{proof}
For $x,y,z \in A$ and $\alpha, \beta,\omega \in \Omega$, by Eq.~(\mref{eq:def[]}), we get $[x,x]_{\omega}=0$ and
\begin{align*}
[p(x), [y,z]_{\alpha}]_{\beta}= & [p(x), y \cdot_{\alpha} z- z \cdot_{\alpha} y]_{\beta}\\
= & p(x) \cdot_{\beta} (y \cdot_{\alpha} z -z \cdot_{\alpha} y) -(y \cdot_{\alpha}z -z \cdot_{\alpha} y) \cdot_{\beta} p(x)\\
= & p(x) \cdot_{\beta}(y \cdot_{\alpha} z)- p(x) \cdot_{\beta}(z \cdot_{\alpha} y) -(y \cdot_{\alpha} z) \cdot_{\beta} p(x) +(z \cdot_{\alpha} y) \cdot_{\beta} p(x).
\end{align*}
Similarly, we have
\begin{align*}
[p(y), [z,x]_{\alpha}]_{\beta}= & p(y) \cdot_{\beta}(z \cdot_{\alpha} x)- p(y) \cdot_{\beta} (x \cdot_{\alpha} z)-(z \cdot_{\alpha} x)\cdot_{\beta} p(y)+ (x \cdot_{\alpha} z) \cdot_{\beta} p(y),\\
[p(z), [x,y]_{\alpha}]_{\beta}= & p(z) \cdot_{\beta}(x \cdot_{\alpha} y)- p(z) \cdot_{\beta} (y \cdot_{\alpha} x)-(x \cdot_{\alpha} y)\cdot_{\beta} p(z)+ (y \cdot_{\alpha} x) \cdot_{\beta} p(z),\\
[p(x), [y,z]_{\beta}]_{\alpha}= & p(x) \cdot_{\alpha}(y \cdot_{\beta} z)- p(x) \cdot_{\alpha} (z \cdot_{\beta} y)-(y \cdot_{\beta} z)\cdot_{\alpha} p(x)+ (z \cdot_{\beta} y) \cdot_{\alpha} p(x),\\
[p(y), [z,x]_{\beta}]_{\alpha}= & p(y) \cdot_{\alpha}(z \cdot_{\beta} x)- p(y) \cdot_{\alpha} (x \cdot_{\beta} z)-(z \cdot_{\beta} x)\cdot_{\alpha} p(y)+ (x \cdot_{\beta} z) \cdot_{\alpha} p(y),\\
[p(z), [x,y]_{\beta}]_{\alpha}= & p(z) \cdot_{\alpha}(x \cdot_{\beta} y)- p(z) \cdot_{\alpha} (y \cdot_{\beta} x)-(x \cdot_{\beta} y)\cdot_{\alpha} p(z)+ (y \cdot_{\beta} x) \cdot_{\alpha} p(z).
\end{align*}
By Eq.~(\mref{eq:chaa}), we get
\begin{align*}
&\ [p(x), [y,z]_{\alpha}]_{\beta}+[p(y), [z,x]_{\alpha}]_{\beta}+[p(z), [x,y]_{\alpha}]_{\beta}\\
&\ +[p(x), [y,z]_{\beta}]_{\alpha}+[p(y), [z,x]_{\beta}]_{\alpha}+[p(z), [x,y]_{\beta}]_{\alpha}= 0.
\end{align*}
Hence $(A, [,]_{\Omega}, p)$ is a compatible Hom-Lie algebra.
\end{proof}

Now we give the definition of matching Hom-preLie algebras.
\begin{defn}
A {\bf matching Hom-preLie algebra} is a \bfk-module $A$ together with a family of binary operations $\ast_{\omega}: A \ot A \rightarrow A, \omega \in \Omega$ and a linear map $p: A \rightarrow A$ such that
\begin{align} \mlabel{eq:mhpa}
p(x) \ast_{\alpha} (y \ast_{\beta} z)-(x \ast_{\alpha} y) \ast_{\beta} p(z)=p(y) \ast_{\beta}(x \ast_{\alpha} z)-(y \ast_{\beta} x)\ast_{\alpha} p(z)
\end{align}
for all $x,y,z \in A$ and $\alpha, \beta \in \Omega$.
\end{defn}
Now we give the relationship between matching Hom-preLie algebras and compatible Hom-Lie algebras.
\begin{prop}
Let $(A, \ast_{\Omega} , p)$ be a matching Hom-preLie algebra. Then $(A, [,]_{\Omega}, p)$
is a compatible Hom-Lie algebra, where
\begin{align}
[,]_{\omega}: A \ot A \rightarrow A, \,\, [x,y]_{\omega}:=x \ast_{\omega} y-y \ast_{\omega} x,\,\, \text{for all $x,y \in A$ and $\omega \in \Omega$.}
\mlabel{eq:def[]1}
\end{align}
\end{prop}
\begin{proof}
For $x,y,z \in A$ and $\alpha, \beta \in \Omega$, by Eq.~(\mref{eq:def[]1}) we have $[x,x]_{\omega}=0$ and
\begin{align*}
[p(x), [y,z]_{\alpha}]_{\beta}= &\ [p(x), y \ast_{\alpha} z- z\ast_{\alpha} y]_{\beta}\\
=&\ p(x) \ast_{\beta}(y \ast_{\alpha} z - z \ast_{\alpha} y)-(y \ast_{\alpha} z-z \ast_{\alpha} y)\ast_{\beta} p(x)\\
=&\ p(x) \ast_{\beta}(y \ast_{\alpha} z) - p(x) \ast_{\beta}(z \ast_{\alpha} y)-(y \ast_{\alpha} z) \ast_{\beta} p(x)+(z \ast_{\alpha} y) \ast_{\beta} p(x).
\end{align*}
Similarly, we have
\begin{align*}
[p(y), [z,x]_{\alpha}]_{\beta}= &\ p(y) \ast_{\beta}(z \ast_{\alpha} x) - p(y) \ast_{\beta}(x \ast_{\alpha} z)-(z \ast_{\alpha} x) \ast_{\beta} p(y)+(x \ast_{\alpha} z) \ast_{\beta} p(y),\\
[p(z), [x,y]_{\alpha}]_{\beta}= &\ p(z) \ast_{\beta}(x \ast_{\alpha} y) - p(z) \ast_{\beta}(y \ast_{\alpha} x)-(x \ast_{\alpha} y) \ast_{\beta} p(z)+(y \ast_{\alpha} x) \ast_{\beta} p(z),\\
[p(x), [y,z]_{\beta}]_{\alpha}= &\ p(x) \ast_{\alpha}(y \ast_{\beta} z) - p(x) \ast_{\alpha}(z \ast_{\beta} y)-(y \ast_{\beta} z) \ast_{\alpha} p(x)+(z \ast_{\beta} y) \ast_{\alpha} p(x),\\
[p(y), [z,x]_{\beta}]_{\alpha}= &\ p(y) \ast_{\alpha}(z \ast_{\beta} x) - p(y) \ast_{\alpha}(x \ast_{\beta} z)-(z \ast_{\beta} x) \ast_{\alpha} p(y)+(x \ast_{\beta} z) \ast_{\alpha} p(y),\\
[p(z), [x,y]_{\beta}]_{\alpha}= &\ p(z) \ast_{\alpha}(x \ast_{\beta} y) - p(z) \ast_{\alpha}(y \ast_{\beta} x)-(x \ast_{\beta} y) \ast_{\alpha} p(z)+(y \ast_{\beta} x) \ast_{\beta} p(z).
\end{align*}
Then by Eq.~(\mref{eq:mhpa}), we get
\begin{align*}
&\ [p(x), [y,z]_{\alpha}]_{\beta}+[p(y), [z,x]_{\alpha}]_{\beta}+[p(z), [x,y]_{\alpha}]_{\beta}\\
&\ +[p(x), [y,z]_{\beta}]_{\alpha}+[p(y), [z,x]_{\beta}]_{\alpha}+[p(z), [x,y]_{\beta}]_{\alpha}=0.
\end{align*}
Hence $(A, [,]_{\Omega}, p)$ is a compatible Hom-Lie algebra.
\end{proof}

\section{Matching Rota-Baxter algebras and Hom-associative algebras}\mlabel{sec2}
In this section, we extend the notion of matching Rota-Baxter algebras to the Hom-associative algebra setting.

\begin{defn}~\mcite{GGZ19}
Let $\lambda_{\Omega}:=(\lambda_{\omega})_{\omega\in \Omega}  \subseteq \bfk$ be a set of scalars indexed by $\Omega$. A {\bf matching Rota-Baxter algebra} of weight $\lambda_{\Omega}$ is an associative algebra $A$ equipped with a family
$P_\Omega:=(P_\omega)_{\omega\in \Omega}$
of linear operators
$P_\omega: R\longrightarrow R, \omega\in \Omega\,,$
that satisfy the matching Rota-Baxter equation
\begin{align}
P_{\alpha}(x) \cdot P_{\beta}(y)=P_{\alpha}(x \cdot P_{\beta}(y))+P_{\beta}(P_{\alpha}(x) \cdot y)+ \lambda_{\beta}P_{\alpha}(x \cdot y), \,\, \text{for all $x,y \in A$ and $\alpha, \beta \in \Omega$.}
\mlabel{eq:mrbe}
\end{align}
\end{defn}

\begin{defn}
 A  {\bf matching Hom-associative Rota-Baxter algebra} is a quadruples $(A, \cdot, P_{\Omega}, p)$, where $(A, P_{\Omega})$ is a matching Rota-Baxter algebra and $(A,\cdot, p)$ is a Hom-associative algebra.
\end{defn}

Taking $p=\id$, we recover to matching Rota-Baxter associative algebras and denote it by $(A, \cdot, P_{\Omega})$. If $\Omega$ is a singleton, a matching Hom-associative Rota-Baxter algebra becomes a Hom-associative Rota-Baxter algebra given in~\mcite{M12}.

A Hom-associative Rota-Baxter algebra can be induced from an associative Rota-Baxter algebra with a particular algebra endomorphism~\mcite{M12}. The following result generalize it to the matching Rota-Baxter case.

\begin{theorem}
Let $(A,\cdot,P_{\Omega})$ be a matching Rota-Baxter algebra and $p: A \rightarrow A$ be an algebra endomorphism which commutes with $P_{\omega}$ for all $\omega \in \Omega$. Then $(A, \cdot_{p}, P_{\Omega}, p)$, where $ x\cdot_{p}y:= p(x \cdot y)$, is a matching Hom-associative Rota-Baxter algebra.
\end{theorem}

\begin{proof}
The Hom-associative structure of the algebra follows from Yau's Theorem in \mcite{Y07}. We only need to show that the matching Rota-Baxter equation holds. For $x,y \in A$ and $\alpha, \beta \in \Omega$,
\begin{align*}
&\ P_{\alpha}(x) \cdot_p P_{\beta}(y)=p \big(P_{\alpha}(x) \cdot P_{\beta}(y) \big) \,\, \text{(by the definition of $\cdot_p$)}\\
=&\  p \big( P_{\alpha}(x \cdot P_{\beta}(y))+ P_{\beta}(P_{\alpha}(x) \cdot y) + \lambda_{\beta} P_{\alpha}(x \cdot y) \big) \,\, \text{(by Eq.~(\mref{eq:mrbe}))}\\
=&\ p \big( P_{\alpha}(x \cdot P_{\beta}(y)) \big)+ p \big( P_{\beta}(P_{\alpha}(x) \cdot y) \big) + \lambda_{\beta} p \big( P_{\alpha}(x \cdot y) \big)\\
=&\ P_{\alpha}\big( p(x \cdot P_{\beta}(y)) \big)+P_{\beta}\big( p(P_{\alpha}(x) \cdot y) \big)+ \lambda P_{\alpha} \big( p(x \cdot y)\big) \,\, \text{(by $p\circ P_{\omega}=P_{\omega} \circ p$)}\\
=&\ P_{\alpha}\big( x \cdot_p P_{\beta}(y) \big) +P_{\beta} \big(P_{\alpha}(x) \cdot_p y  \big)+ \lambda P_{\alpha} \big(x \cdot_p y \big),
\end{align*}
as required.
\end{proof}

Given a matching Hom-associative Rota-Baxter algebra $(A, \cdot, P_{\Omega}, p)$, it is natural to wonder that whether this matching Hom-associative Rota-Baxter algebra is induced by an ordinary associative matching Rota-Baxter algebra $(A, \cdot', P_{\Omega})$, i.e. $p$ is an algebra endomorphism with respect to $\cdot'$ and $\cdot = p\circ \cdot'$.

Let $(A, \cdot, p)$ be a multiplicative Hom-associative algebra, i.e. $p(a \cdot b)=p(a) \cdot p(b)$ for all $a,b \in A$. It was proved in~\mcite{G09} that in case $p$ is invertible, $(A, p^{-1} \circ \cdot)$ an associative algebra. It is generalized to the multiplicative Hom-associative Rota-Baxter algebras in~\mcite{M12} and the following result generalize it to the multiplicative matching Hom-associative Rota-Baxter algebras.

\begin{prop}
Let $(A, \cdot, P_{\Omega}, p)$ be a multiplicative matching Hom-assoicative Rota-Baxter algebra, where $p$ is invertible and $p \circ P_{\omega}=P_{\omega} \circ p$ for each $\omega \in \Omega$. Then $(A, \cdot':= p^{-1} \circ \cdot, P_{\Omega})$ is an associative matching Rota-Baxter algebra.
\end{prop}

\begin{proof}
For $x,y,z \in A$, we have
\begin{align*}
&\ (x \cdot' y) \cdot' z- x \cdot' (y \cdot' z)\\
=&\  p^{-1} \big( p^{-1}(x \cdot y) \cdot z \big)- p^{-1} \big(x \cdot p^{-1}(y \cdot z)\big) \,\, \text{(by $\cdot'=p^{-1} \circ \cdot$)}\\
=&\ p^{-2} \big((x \cdot y) \cdot p(z)- p(x) \cdot (y \cdot z)\big) \,\, \text{(by $p(x) \cdot p(y)=p(x \cdot y)$)}\\
=&\ 0.
\end{align*}
Hence the associativity condition holds. For $\alpha, \beta \in \Omega$, we have
\begin{align*}
&\ P_{\alpha}(x) \cdot' P_{\beta}(y)\\
=&\ p^{-1} \big(P_{\alpha}(x) \cdot P_{\beta}(y)\big)\\
=&\ p^{-1} \big(P_{\alpha}(x \cdot P_{\beta}(y))+ P_{\beta}(P_{\alpha}(x) \cdot y)+\lambda_{\beta} P_{\alpha}(x \cdot y)  \big)\\
=&\ P_{\alpha} \big(p^{-1}(x \cdot P_{\beta}(y)) \big) +P_{\beta} \big(p^{-1}(P_{\alpha}(x) \cdot y)\big)+ \lambda_{\beta} P_{\alpha} \big(p^{-1}(x \cdot y) \big)\\
=&\ P_{\alpha}\big(x \cdot' P_{\beta}(y) \big)+ P_{\beta} \big(P_{\alpha}(x) \cdot' y \big)+ \lambda_{\beta} P_{\alpha} (x \cdot'y).
\end{align*}
Hence the matching Rota-Baxter equation holds for the new multiplication and $(A, \cdot', P_{\Omega})$ is an associative matching Rota-Baxter algebra.
\end{proof}

There are two new ways of constructing Hom-associative algebras from a given multiplicative Hom-associative algebra~\mcite{M12, Y10}.

\begin{defn}(\mcite{M12,Y10}) Let $(A,\cdot,p)$ be a multiplicative Hom-algebra and $n\geq 0$. Then the following two algebras are also Hom-associative algebras:
\begin{enumerate}
\item the $n$-th derived Hom-algebra of type $1$ of $A$ defined by
\begin{align*}
 A^{n} = (A,\cdot^{(n)}=p^n \circ \cdot, p^{n+1}),
\end{align*}
\item the $n$-th derived Hom-algebra of type $2$ of $A$ defined by
\begin{align*}
A^n=(A, \cdot^{(n)}=p^{2^n-1}\circ \cdot, p^{2^n}).
\end{align*}
\end{enumerate}
\end{defn}

Now we show that the $n$-th derived Hom-algebra of type 1 and 2 of a multiplicative matching Hom-associative Rota-Baxter algebra is also a matching Hom-associative Rota-Baxter algebra generalizing the Rota-Baxter case in~\mcite{M12}.

\begin{theorem} Let $(A, \cdot, P_{\Omega},p)$ be a multiplicative matching Hom-associative Rota-Baxter algebra such that $p \circ P_{\omega}=P_{\omega} \circ p$ for all $\omega \in \Omega$. Then
\begin{enumerate}
\item \mlabel{Item13} the $n$-th derived Hom-algebra of type $1$ $(A,\cdot^{(n)}=p^n \circ \cdot, p^{n+1})$ is a matching Hom-associative Rota-Baxter algebra.

\item \mlabel{Item23} the $n$-th derived Hom-algebra of type 2 $(A, \cdot^{(n)}=p^{2^n-1}\circ \cdot, p^{2^n})$ is a matching Hom-associative Rota-Baxter algebra.
\end{enumerate}
\end{theorem}

\begin{proof}
(\mref{Item13}) By~\mcite{Y10}, $(A, \cdot^n, p^{n+1})$ is a Hom-associative algebra. Now we show the matching Rota-Baxter equation holds.
For $x, y, z \in A$ and $\alpha, \beta \in \Omega$, we have
\begin{align*}
&\ P_{\alpha}(x) \cdot^n P_{\beta}(y)\\
=&\ p^n \big( P_{\alpha}(x) \cdot P_{\beta}(y)\big)\\
=&\ p^n \big(P_{\alpha}(x \cdot P_{\beta}(y))+P_{\beta}(P_{\alpha}(x) \cdot y)+ \lambda_{\beta} P_{\alpha}(x \cdot y) \big)\\
=&\ P_{\alpha} \big(p^n(x \cdot P_{\beta}(y))\big) +P_{\beta} \big(p^n(P_{\alpha}(x) \cdot y)\big) + \lambda_{\beta} P_{\alpha} \big(p^n(x \cdot y) \big)\\
=&\ P_{\alpha} \big(x \cdot^n P_{\beta}(y) \big)+ P_{\beta} \big(P_{\alpha}(x) \cdot^n y \big)+ \lambda_{\beta} P_{\alpha} (x \cdot^n y).
\end{align*}
Thus the matching Rota-Baxter equation holds for the new multiplication.

(\mref{Item23}) By~\mcite{Y10}, $(A, \cdot^{(n)}=p^{2^n-1}\circ \cdot, p^{2^n})$ is also a Hom-associative algebra. For $x,y,z \in A$ and $\alpha, \beta \in \Omega$, we have
\begin{align*}
&\ P_{\alpha}(x) \cdot^n P_{\beta}(y)\\
=&\ p^{2^n-1} \big( P_{\alpha}(x) \cdot P_{\beta}(y)\big)\\
=&\ p^{2^n-1} \big(P_{\alpha}(x \cdot P_{\beta}(y))+P_{\beta}(P_{\alpha}(x) \cdot y)+ \lambda_{\beta} P_{\alpha}(x \cdot y) \big)\\
=&\ P_{\alpha} \big(p^{2^n-1}(x \cdot P_{\beta}(y))\big) +P_{\beta} \big(p^{2^n-1}(P_{\alpha}(x) \cdot y)\big) + \lambda_{\beta} P_{\alpha} \big(p^{2^n-1}(x \cdot y) \big)\\
=&\ P_{\alpha} \big(x \cdot^n P_{\beta}(y) \big)+ P_{\beta} \big(P_{\alpha}(x) \cdot^n y \big)+ \lambda_{\beta} P_{\alpha} (x \cdot^n y).
\end{align*}
This completes the proof.
\end{proof}

Let $(A,\cdot)$ be an associative algebra. The centroid of A is defined by
\begin{align*}
Cent(A):= \Big\{ p \in End(A) \,\big{|}\, p(x \cdot y)= p(x) \cdot y= x \cdot p(y)\, \text{ for all $x,y \in A$} \Big\}.
\end{align*}
The same definition of the centroid is assumed for Hom-associative algebras.

In~\mcite{BM14}, Benayadi and Makhlouf gave the construction of Hom-algebras using elements of the centroid for Lie algebras. In~\mcite{M12}, the construction was extended to Rota-Baxter algebras. Now we generalize it to the matching Rota-Baxter case.

\begin{prop}
Let $(A, \cdot, P_{\Omega})$ be an associative matching Rota-Baxter algebra. For $p \in Cent(A)$ and $x, y \in A$, define
\begin{align*}
x \cdot_p^1 y :=p(x) \cdot y \,\, \text{and} \,\,  x \cdot_p^2 y := p(x) \cdot p(y).
\end{align*}
If $p \circ P_{\omega} =P_{\omega} \circ p$ for all $\omega \in \Omega$, then $(A, \cdot_p^1, P_{\Omega}, p)$ and $(A, \cdot_p^2, P_{\Omega}, p)$ are matching Hom-associative Rota-Baxter algebras.
\end{prop}

\begin{proof}
By~\mcite{M12}, $(A, \cdot_p^1, p)$ and $(A, \cdot_p^2, p)$ are Hom-associative algebras. Now we show that they are also matching Rota-Baxter algebras.
For $x,y \in A$ and $\alpha, \beta \in \Omega$, we have

\begin{align*}
&\ P_{\alpha}(x) \cdot_p^1 P_{\beta}(y)= p(P_{\alpha}(x)) \cdot P_{\beta}(y)= P_{\alpha}(p(x)) \cdot P_{\beta}(y)\\
=&\ P_{\alpha} \big(p(x) \cdot P_{\beta}(y) \big)+ P_{\beta}\big(P_{\alpha}(p(x)) \cdot y \big)+ \lambda_{\beta} P_{\alpha}(p(x) \cdot y)\\
=&\ P_{\alpha} \big(x \cdot_p^1 P_{\beta}(y) \big)+ P_{\beta} \big(P_{\alpha}(x) \cdot_p^1 y \big)+ \lambda_{\beta} P_{\alpha} (x \cdot_p^1 y)
\end{align*}
and
\begin{align*}
&\ P_{\alpha}(x) \cdot_p^2 P_{\beta}(y)=p(P_{\alpha}(x)) \cdot p(P_{\beta}(y))=P_{\alpha}(p(x)) \cdot P_{\beta}(p(y))\\
=&\ P_{\alpha}\big(p(x) \cdot P_{\beta}(p(y))\big)+P_{\beta}\big(P_{\alpha}(p(x)) \cdot p(y) \big)+ \lambda_{\beta} P_{\alpha}(p(x) \cdot p(y))\\
=&\ P_{\alpha} \big(p(x) \cdot p(P_{\beta}(y)) \big)+P_{\beta} \big(p(P_{\alpha}(x)) \cdot p(y) \big)+\lambda_{\beta} P_{\alpha}(p(x) \cdot p(y))\\
=&\ P_{\alpha} \big(x \cdot_p^2 P_{\beta}(y) \big)+P_{\beta} \big(P_{\alpha}(x) \cdot_p^2 y \big)+ \lambda_{\beta} P_{\alpha}(x \cdot_p^2 y).
\end{align*}
This completes the proof.
\end{proof}

\section{Matching Hom-dendriform algebras and matching Hom-tridendriform algebras}\mlabel{sec3}
In this section, we introduce the notions of matching Hom-dendriform algebras and matching Hom-tridendriform algebras generalizing the definitions  of matching dendriform algebras and matching tridendriform algebras given in~\mcite{GGZ19}.

\begin{defn}
A {\bf matching Hom-dendriform algebra} is a \bfk-module $D$ together with a family of binary operations $\odot_{\omega}: D \ot D \rightarrow D$, where $\odot \in \{\prec, \succ \}$ and $\omega \in \Omega$, and a linear map $p: D \rightarrow D$ such that for all $x,y,z \in D$ and $\alpha, \beta \in \Omega$,
\begin{align*}
(x \prec_{\alpha} y) \prec_{\beta} p(z)= &\ p(x) \prec_{\alpha} (y \prec_{\beta} z)+ p(x) \prec_{\beta} (y \succ_{\alpha} z),\\
(x \succ_{\alpha} y) \prec_{\beta} p(z)= &\ p(x) \succ_{\alpha} (y \prec_{\beta} z),\\
(x \prec_{\beta} y) \succ_{\alpha} p(z)+ (x \succ_{\alpha} y) \succ_{\beta} p(z) = &\ p(x) \succ_{\alpha} (y \succ_{\beta} z).
\end{align*}
For simplicity, we denote it by $(D, \prec_{\Omega}, \succ_{\Omega}, p)$.
\end{defn}

\begin{defn}
A {\bf matching Hom-tridendriform algebra} is a \bfk-module $D$ together with a family of binary operations $\odot_{\omega}: D \ot D \rightarrow D$, where $\odot \in \{\prec,\bullet, \succ \}$ and $\omega \in \Omega$, and a linear map $p: D \rightarrow D$ such that for all $x,y,z \in D$ and $\alpha, \beta \in \Omega$,
\begin{align}
(x \prec_{\alpha} y) \prec_{\beta} p(z) = &\ p(x) \prec_{\alpha} (y \prec_{\beta} z)+ p(x) \prec_{\beta} (y \succ_{\alpha} z)+ p(x) \prec_{\alpha} (y \bullet_{\beta} z), \mlabel{eq:mhta1}\\
(x \succ_{\alpha} y) \prec_{\beta} p(z) = &\ p(x) \succ_{\alpha} (y \prec_{\beta} z),\mlabel{eq:mhta2} \\
p(x) \succ_{\alpha} (y \succ_{\beta} z) = &\ (x \prec_{\beta} y) \succ_{\alpha} p(z) + (x \succ_{\alpha} y) \succ_{\beta} p(z)+(x \bullet_{\beta} y) \succ_{\alpha} p(z), \mlabel{eq:mhta3}\\
(x \succ_{\alpha} y) \bullet_{\beta} p(z)= &\ p(x) \succ_{\alpha} (y \bullet_{\beta} z), \mlabel{eq:mhta4}\\
(x \prec_{\alpha} y) \bullet_{\beta} p(z)= &\ p(x) \bullet_{\beta}(y \succ_{\alpha} z), \mlabel{eq:mhta5}\\
(x \bullet_{\alpha} y) \prec_{\beta} p(z)=&\ p(x) \bullet_{\alpha} (y \prec_{\beta} z), \mlabel{eq:mhta6}\\
(x \bullet_{\alpha} y) \bullet_{\beta} p(z) =&\ p(x) \bullet_{\alpha} (y \bullet_{\beta} z). \mlabel{eq:mhta7}
\end{align}
\end{defn}

\begin{defn}
\begin{enumerate}
  \item Let $(D, \prec_{\Omega}, \succ_{\Omega}, p)$ and $(D', \prec'_{\Omega}, \succ'_{\Omega} , p')$ be two matching Hom-dendriform algebras. A linear map $f:D \rightarrow D'$ is called a {\bf matching Hom-dendriform algebra morphism} if for all $\omega \in \Omega$,
      \begin{align*}
      \prec'_{\omega} \circ ~(f \ot f)=f \circ \prec_{\omega}, \,\, \succ_{\omega} \circ~(f \ot f)=f \circ \succ_{\omega} \,\, \text{and} \,\, p'\circ f=f \circ p.
      \end{align*}
  \item Let $(D, \prec_{\Omega}, \bullet_{\Omega}, \succ_{\Omega}, p)$ and $(D', \prec'_{\Omega},\bullet_{\Omega}, \succ'_{\Omega}, p')$ be two matching Hom-tridendriform algebras. A linear map $f:D \rightarrow D'$ is called a {\bf matching Hom-tridendriform algebra morphism} if for all $\omega \in \Omega$,
      \begin{align*}
      \prec'_{\omega} \circ ~(f \ot f)=f \circ \prec_{\omega}, \,\,\bullet'_{\omega} \circ (f \ot f)=f \circ \bullet'_{\omega}, \,\, \succ_{\omega} \circ ~(f \ot f)=f \circ \succ_{\omega} \,\, \text{and} \,\, p'\circ f=f \circ p.
      \end{align*}
\end{enumerate}
\end{defn}

The following results show that we can construct a matching Hom-(tri)dendriform algebra from a matching (tri)dendriform algebra, generalizing the (tri)dendriform case in~\mcite{M12}.

\begin{theorem}
\begin{enumerate}
 \item \mlabel{Item11}Let $(D,\prec_{\Omega}, \succ_{\Omega})$ be a matching dendriform algebra and $p: D \rightarrow D$ be a matching dendriform algebra endomorphism. Then $A_p = (A, \prec_{p,\Omega}, \succ_{p,\Omega},p)$, where $\prec_{p,\omega}:=p\circ \prec_{\omega}$ and $\succ_{p,\omega}:=p \circ \succ_{\omega}$ for each $\omega \in \Omega$, is a matching Hom-dendriform algebra. Moreover, suppose that $(A', \prec'_{\Omega}, \succ'_{\Omega})$ is another matching dendriform algebra and $p':A' \rightarrow A'$ is a matching dendriform algebra endomorphism. If $f: A \rightarrow A'$ is a matching dendriform algebra morphism that satisfies $f \circ p =p' \circ f$, then
\begin{align*}
f: (D,\prec_{p, \Omega},\succ_{p, \Omega},p)\rightarrow (D',\prec'_{p, \Omega},\succ'_{p, \Omega},p')
\end{align*}
is a morphism of matching Hom-dendriform algebras.
\item \mlabel{Item22} Let $(D,\prec_{\Omega},\bullet_{\Omega}, \succ_{\Omega})$ be a matching tridendriform algebra and $p: D \rightarrow D$ be a matching tridendriform algebra endomorphism. Then $A_p = (A,\prec_{p,\Omega},\bullet_{p,\Omega}, \succ_{p,\Omega},p)$, where $\prec_{p,\omega}:=p\circ \prec_{\omega}$, $\bullet_{p, \omega}:= p \circ \bullet_{\omega}$ and $\succ_{p,\omega}=p \circ \succ_{\omega}$ for each $\omega \in \Omega$, is a matching Hom-tridendriform algebra. Moreover, suppose that $(A', \prec'_{\Omega},\bullet'_{\Omega}, \succ'_{\Omega})$ is another matching tridendriform algebra and $p':A' \rightarrow A'$ is a matching tridendriform algebra endomorphism. If $f: A \rightarrow A'$ is a matching tridendriform algebra morphism that satisfies $f \circ p =p' \circ f$, then
\begin{align*}
f: (D,\prec_{p, \Omega},\bullet_{p,\Omega},\succ_{p, \Omega} ,p)\rightarrow (D',\prec'_{p, \Omega},\bullet'_{p,\Omega}, \succ'_{p, \Omega},p')
\end{align*}
is a morphism of matching Hom-tridendriform algebras.
\end{enumerate}
\end{theorem}

\begin{proof}
We just prove Item~(\mref{Item22}) and Item~(\mref{Item11}) can be proved similarly.
For any $x, y, z \in A$ and $\alpha, \beta \in \Omega$, we have
\begin{align*}
(x \prec_{p, \alpha} y) \prec_{p, \beta} p(z)=p\big(p(x \prec_{\alpha} y) \prec_{\beta} p(z) \big)&=p^2 \big((x \prec_{\alpha} y) \prec_{\beta} z \big);\\
p(x) \prec_{p, \alpha} (y \prec_{p, \beta} z)= p\big(p(x) \prec_{\alpha} p(y \prec_{\beta} z)\big)&= p^2\big(x \prec_{\alpha} (y \prec_{\beta} z)\big);\\
p(x) \prec_{p, \beta} (y \succ_{p, \alpha} z)= p\big(p(x) \prec_{\beta} p(y \succ_{\alpha} z)\big)&=p^2\big(x \prec_{\beta}( y \succ_{\alpha} z)\big);\\
p(x) \prec_{p,\alpha} (y \bullet_{p,\beta} z)= p \big(p(x) \prec_{\alpha} p(y \bullet_{\beta} z) \big)&= p^2 \big(x \prec_{\alpha} (y \bullet_{\beta} z) \big).
\end{align*}
Hence
\begin{align*}
(x \prec_{p, \alpha} y) \prec_{p, \beta} p(z)= p(x) \prec_{p, \alpha} (y \prec_{p, \beta} z)+ p(x) \prec_{p, \beta} (y \succ_{p, \alpha} z)+p(x)\prec_{p,\alpha}(y \bullet_{p,\beta} z),
\end{align*}
that is, Eq.~(\ref{eq:mhta1}) holds for $(A,\, \prec_{p,\Omega},\, \bullet_{p,\Omega}, \, \succ_{p,\Omega} ,\, p)$. Similarly, Eqs.~(\mref{eq:mhta2})-(\mref{eq:mhta7}) holds. Hence $(A, \, \prec_{p,\Omega},\, \bullet_{p,\Omega},\,  \succ_{p,\Omega},\, p)$ is a matching Hom-tridendriform algebra.
And
\begin{align*}
f(x) \prec_{p', \alpha} f(y)=p'(f(x) \prec_{\alpha} f(y))=p' \circ f(x \prec_{\alpha} y)= f\circ p (x \prec_{\alpha} y)&=f(x \prec_{p,\alpha} y);\\
f(x) \succ_{p', \alpha} f(y)=p'(f(x) \succ_{\alpha} f(y))=p' \circ f(x \succ_{\alpha} y)=f \circ p(x \succ_{\alpha} y)&= f(x \succ_{p, \alpha} y);\\
f(x) \bullet_{p', \alpha} f(y)=p'(f(x) \bullet_{\alpha} f(y))= p' \circ f(x \bullet_{\alpha} (y))=f \circ p(x \bullet y)&=f(x \bullet_{p,\alpha} y).
\end{align*}
Hence $f: (D,\, \prec_{p, \Omega},\, \bullet_{p,\Omega},\, \succ_{p, \Omega},p)\rightarrow (D',\, \prec'_{p, \Omega},\, \bullet'_{p,\Omega},\, \succ'_{p, \Omega},\, p')$ is a morphism of matching Hom-tridendriform algebras.
\end{proof}

Now we show that any linear combinations of the operations of a matching Hom-dendriform algebra still result in a matching Hom-dendriform algebra, generalizing the matching dendriform case in~\mcite{GGZ19}.

\begin{prop}
Let $I$ be an nonempty set. For each $i \in I$, let $A_i: \Omega \rightarrow k$ be a map with finite supports, identified with finite set $A_{i}=(a_{i,\omega})_{\omega\in \Omega},  a_{i,\omega}\in \bfk $.
\begin{enumerate}
  \item \mlabel{Item1} Let $(D, \prec_{\Omega}, \succ_{\Omega},p)$ be a matching Hom-dendriform algebra. Define the following binary operations:
      \begin{align*}
      \odot_i := \sum \limits_{\omega \in \Omega} a_{i,\omega} \odot_{\omega}, \,\, \text{where $\odot \in \{\prec, \succ \}$ and $i \in I$.}
      \end{align*}
      Then $(D, \prec_I, \succ_I,p)$ is also a matching Hom-dendriform algebra.
  \item \mlabel{Item2} Let $(T, \prec_{\Omega}, \bullet_{\Omega}, \succ_{\Omega},p)$ be a matching Hom-tridendriform algebra. Define the following binary operations:
      \begin{align*}
      \odot_i := \sum \limits_{\omega \in \Omega} a_{i,\omega} \odot_{\omega}, \,\, \text{where $\odot \in \{\prec, \bullet, \succ \}$ and $i \in I$.}
      \end{align*}
      Then $(T, \prec_I, \bullet_I, \succ_I,p)$ is also a matching Hom-tridendriform algebra.
\end{enumerate}
\end{prop}

\begin{proof}
We just prove Item~(\mref{Item2}) and Item~(\mref{Item1}) can be proved similarly. For $x,y,z \in D$ and $i,j \in I$, we have
\begin{align*}
&\ (x \prec_i y)\prec_j p(z)= \sum\limits_{\beta \in \Omega} b_{j,\beta} (\sum\limits_{\alpha \in \Omega} a_{i,\alpha} x \prec_{\alpha} y) \prec_{\beta} p(z)= \sum\limits_{\alpha \in \Omega}\sum\limits_{\beta \in \Omega} a_{i,\alpha} b_{j,\beta} (x \prec_{\alpha} y) \prec_{\beta} p(z)\\
=&\ \sum\limits_{\alpha \in \Omega}\sum\limits_{\beta \in \Omega} a_{i,\alpha} b_{j,\beta} (p(x) \prec_{\alpha}(y \prec_{\beta} z)+ p(x) \prec_{\beta} (y \succ_{\alpha} z)+ p(x) \prec_{\alpha} (y \bullet_{\beta} z))\\
=&\ \sum\limits_{\alpha \in \Omega} a_{i,\alpha} p(x) \prec_{\alpha} (\sum\limits_{\beta \in \Omega} b_{j,\beta} y \prec_{\beta} z)+ \sum\limits_{\beta \in \Omega} b_{j,\beta} p(x) \prec_{\beta} (\sum\limits_{\alpha \in \Omega} a_{i,\alpha} y \succ_{\alpha} z)+ \sum\limits_{\alpha \in \Omega} a_{i,\alpha} p(x) \prec_{\alpha}(\sum\limits_{\beta \in \Omega} b_{j, \beta} y \bullet_{\beta} z)\\
=&\ \sum\limits_{\alpha \in \Omega} a_{i,\alpha} p(x) \prec_{\alpha}(y \prec_j z)+ \sum\limits_{\beta \in \Omega} b_{j,\beta} p(x) \prec_{\beta}(y\succ_i z)+ \sum\limits_{\alpha \in \Omega} a_{i,\alpha} p(x) \prec_{\alpha} (y \bullet_{j} z)\\
=&\ p(x) \prec_i(y \prec_j z)+p(x) \prec_j(y \succ_i z)+p(x) \prec_i(y\bullet_j z).
\end{align*}
Hence, Eq.~(\mref{eq:mhta1}) holds. Similarly, Eqs.~(\mref{eq:mhta2})-(\mref{eq:mhta7}) hold. Hence $(T, \prec_I, \bullet_I, \succ_I, p)$ is a matching Hom-tridendriform algebra.
\end{proof}

The following results establish the connections between matching Hom-(tri)dendriform algebras and compatible Hom-associative algebras, generalizing the well known result that a (tri) dendriform algebra has an associative algebraic structure.

\begin{theorem}
\begin{enumerate}
\item \mlabel{Item111} Let $(A, \prec_{\Omega}, \succ_{\Omega},p)$ be a matching Hom-dendriform algebra. Then $(A, \cdot_{\Omega}, p)$ is a compatible Hom-associative algebra, where
\begin{align*}
\cdot_{\omega}: A \ot A \rightarrow A, \,\, x \cdot_{\omega} y:= x \prec_{\omega}y+ x \succ_{\omega} y \,\, \text{for $x,y \in A$ and $\omega \in \Omega$.}
\end{align*}
\item \mlabel{Item222} Let $(A, \prec_{\Omega}, \bullet_{\Omega}, \succ_{\Omega},p)$ be a matching Hom-tridendriform algebra. Then $(A,  \cdot_{\Omega}, p)$ is a compatible Hom-associative algebra, where
\begin{align*}
\cdot_{\omega}: A \ot A \rightarrow A, \,\, x \cdot_{\omega} y:= x \prec_{\omega}y+ x \bullet_{\omega} y+ x \succ_{\omega} y \,\, \text{for $x,y \in A$ and $\omega \in \Omega$.}
\end{align*}
\end{enumerate}
\end{theorem}

\begin{proof}
We only prove Item~(\mref{Item2}) and Item~(\mref{Item111}) can be proved similarly.
For $x,y,z \in A$ and $\alpha, \beta \in \Omega$, we have
\begin{align*}
&\ (x \cdot_{\alpha} y) \cdot_{\beta} p(z)+(x \cdot_{\beta} y) \cdot_{\alpha} p(z)\\
=&\ (x \prec_{\alpha} y+ x\bullet_{\alpha} y+ x \succ_{\alpha} y) \cdot_{\beta} p(z)+(x \prec_{\beta} y + x \bullet_{\beta}y+ x \succ_{\beta} y) \cdot_{\alpha}p(z)\\
=&\ (x \prec_{\alpha} y) \prec_{\beta} p(z)+(x \bullet_{\alpha} y) \prec_{\beta}p(z)+(x\succ_{\alpha} y)\prec_{\beta}p(z)+(x \prec_{\alpha} y)\bullet_{\beta} p(z)+(x \bullet_{\alpha}y)\bullet_{\beta} p(z)\\
&\ +(x \succ_{\alpha} y) \bullet_{\beta}p(z)+(x\prec_{\alpha} y) \succ_{\beta}p(z)+(x \bullet_{\alpha} y) \succ_{\beta} p(z)+(x \succ_{\alpha} y)\succ_{\beta} p(z)+(x \prec_{\beta} y)\prec_{\alpha}p(z)\\
&\ +(x\succ_{\beta} y) \prec_{\alpha} p(z)+(x\bullet_{\beta} y) \prec_{\alpha}p(z)+(x \prec_{\beta} y)\bullet_{\alpha} p(z)+(x \succ_{\beta} y)\bullet_{\alpha} p(z)+(x \bullet_{\beta} y)\bullet_{\alpha} p(z)\\
&\ +(x \prec_{\beta} y) \succ_{\alpha} p(z)+ (x \bullet_{\beta} y)\succ_{\alpha} p(z)+(x\succ_{\beta} y) \succ_{\alpha} p(z),
\end{align*}
and
\begin{align*}
&\ p(x) \cdot_{\alpha}(y \cdot_{\beta} z)+p(x) \cdot_{\beta}(y \cdot_{\alpha} z)\\
=&\ p(x) \cdot_{\alpha} (y \prec_{\beta} z+ y \bullet_{\beta} z+ y \succ_{\beta} z)+p(x) \cdot_{\beta}(y \prec_{\alpha} z+ y \bullet_{\alpha} z+ y \succ_{\alpha} z)\\
=&\ p(x)\prec_{\alpha}(y\prec_{\beta} z)+ p(x) \prec_{\alpha} (y \bullet_{\beta} z)+ p(x) \prec_{\alpha} (y \succ_{\beta} z) +p(x) \bullet_{\alpha} (y \prec_{\beta} z)+p(x) \bullet_{\alpha} (y \bullet_{\beta} z)\\
&\ +p(x) \bullet_{\alpha}(y \succ_{\beta} z)+p(x)\succ_{\alpha}(y \prec_{\beta} z)+p(x)\succ_{\alpha}(y \bullet_{\beta} z)+ p(x)\succ_{\alpha}(y \succ_{\beta} z)+p(x) \prec_{\beta}(y \prec_{\alpha} z)\\
&\ +p(x)\prec_{\beta}(y \bullet_{\alpha} z)+p(x) \prec_{\beta}(y \succ_{\alpha} z)+p(x)\bullet_{\beta}(y \prec_{\alpha} z)+p(x) \bullet_{\beta}(y \bullet_{\alpha} z)+p(x) \bullet_{\beta}(y \succ_{\alpha} z)\\
&\ +p(x) \succ_{\beta} (y \prec_{\alpha} z)+p(x)\succ_{\beta}(y \bullet_{\alpha} z)+p(x) \succ_{\beta}(y \succ_{\alpha} z).
\end{align*}
By Eqs~(\mref{eq:mhta1}-\mref{eq:mhta7}), we get
\begin{align*}
&\ (x \cdot_{\alpha} y) \cdot_{\beta} p(z)+(x \cdot_{\beta} y) \cdot_{\alpha} p(z)= p(x) \cdot_{\alpha}(y \cdot_{\beta} z)+p(x) \cdot_{\beta}(y \cdot_{\alpha} z).
\end{align*}
Hence $(A, \cdot_{\Omega}, p)$ is a compatible Hom-associative algebra.
\end{proof}

Now we explore the relationship between matching Hom-dendriform algebras and matching Hom-preLie algebras.
\begin{theorem}\mlabel{thm:dentopre}
Let $(A, \prec_{\Omega}, \succ_{\Omega}, p)$ be a matching Hom-dendriform algebra. Then $(A,  \ast_{\Omega}, p)$ is a matching Hom-preLie algebra, where
\begin{align*}
\ast_{\omega}: A \ot A \rightarrow A, \,\, x \ast_{\omega}y :=x \succ_{\omega}y-y \prec_{\omega} x \, \, \text{ for $x,y \in A$ and $\omega \in \Omega$.}
\end{align*}
\end{theorem}

\begin{proof}
For $x,y,z \in A$ and $\alpha, \beta \in \Omega$, we have
\begin{align*}
&\ p(x) \ast_{\alpha}(y \ast_{\beta} z)-(x\ast_{\alpha} y) \ast_{\beta}p(z)\\
=&\ p(x) \ast_{\alpha}(y \succ_{\beta} z- z\prec_{\beta} y)-(x\succ_{\alpha}y - y \prec_{\alpha} x) \ast_{\beta} p(z)\\
=&\ p(x)\succ_{\alpha}(y \succ_{\beta} z)-p(x)\succ_{\alpha}(z \prec_{\beta} y)-(y\succ_{\beta} z)\prec_{\alpha}p(x)+(z \prec_{\beta} y)\prec_{\alpha}p(x)-(x\succ_{\alpha} y)\succ_{\beta}p(z)\\
&\ +(y \prec_{\alpha} x)\succ_{\beta} p(z)+p(z)\prec_{\beta}(x \succ_{\alpha} y)-p(z)\prec_{\beta}(y\prec_{\alpha} x)
\end{align*}
and
\begin{align*}
&\ p(y)\ast_{\beta} (x \ast_{\alpha} z)-(y \ast_{\beta} x)\ast_{\alpha} p(z)\\
=&\ p(y) \ast_{\beta}(x \succ_{\alpha} z- z\prec_{\alpha} x)-(y \succ_{\beta} x- x\prec_{\beta} y) \ast_{\alpha} p(z)\\
=&\ p(y)\succ_{\beta}(x\succ_{\alpha}z)-p(y)\succ_{\beta}(z \prec_{\alpha} x)-(x\succ_{\alpha} z)\prec_{\beta} p(y)+(z \prec_{\alpha}x)\prec_{\beta} p(y)-(y\succ_{\beta} x)\succ_{\alpha}p(z)\\
&\ +(x\prec_{\beta} y)\succ_{\alpha} p(z)+p(z)\prec_{\alpha}(y \succ_{\beta} x)-p(z)\prec_{\alpha}(x\prec_{\beta} y).
\end{align*}
By Eqs~(\mref{eq:mhta1})-(\mref{eq:mhta7}) , we get
\begin{align*}
p(x) \ast_{\alpha}(y \ast_{\beta} z)-(x\ast_{\alpha} y) \ast_{\beta}p(z)=p(y)\ast_{\beta} (x \ast_{\alpha} z)-(y \ast_{\beta} x)\ast_{\alpha} p(z).
\end{align*}
Hence  $(A,  \ast_{\Omega}, p)$ is a matching Hom-preLie algebra.
\end{proof}

A matching Rota-Baxter algebra $(A, \cdot, P_{\Omega})$ is of weight 0 if the set $\lambda_{\Omega}=\{0\}$. The connections between Rota-Baxter algebras and (tri)dendriform algebras are given in~\mcite{A01, E02} and extended to matching Rota-Baxter algebras. Now we generalize it to matching Hom-associative Rota-Baxter algebra.

\begin{prop}
\begin{enumerate}
\item \mlabel{Item14}  Let $(A,\cdot,P_{\Omega},p)$ be a matching Hom-associative Rota-Baxter algebra of weight $0$. Assume that $p \circ P_{\omega}=P_{\omega} \circ p$ for each $\omega \in \Omega$. Define the operations $\prec_{\omega}$ and $\succ_{\omega}$ for $\omega \in \Omega$ by
\begin{align*}
 x \prec_{\omega} y := x \cdot P_{\omega}(y) \,\, \text{and} \,\,  x \succ_{\omega} y = P_{\omega}(x)\cdot y, \,\, \text{for $x,y \in A$}.
\end{align*}
Then $(A, \prec_{\Omega}, \succ_{\Omega},p)$ is a matching Hom-dendriform algebra.
\item \mlabel{Item24} Let $(A, \cdot, P_{\Omega},p)$ be a matching Hom-associative Rota-Baxter algebra. Assume that $p \circ P_{\omega} =P_{\omega} \circ p$ for each $\omega \in \Omega$. Define the operations $\prec_{\omega}$, $\succ_{\omega}$, $\omega \in \Omega$  by
\begin{align*}
 x \prec_{\omega} y := x \cdot P_{\omega}(y) + \lambda_{\omega} x \cdot y \,\,  \text{and} \,\,  x \succ_{\omega} y = P_{\omega}(x) \cdot y, \,\, \text{for $x,y \in A$}.
 \end{align*}
Then $(A,\prec_{\Omega}, \succ_{\Omega},p)$ is a matching Hom-dendriform algebra.
\end{enumerate}
\mlabel{prop:rbtoden}
\end{prop}

\begin{proof}
Since Item~(\mref{Item14}) can be seen as a special case of Item~(\mref{Item24}) by taking $\lambda_{\Omega}=\{0\}$, we only prove Item~(\mref{Item24}).
For $x, y, z \in A$ and $\alpha, \beta \in \Omega$, we have
\begin{align*}
&\ p(x) \prec_{\alpha} (y \prec_{\beta} z)+ p(x) \prec_{\beta}( y \succ_{\alpha} z)\\
=&\ p(x) \prec_{\alpha} \big(y \cdot P_{\beta}(z) + \lambda_{\beta} y \cdot z \big)+ p(x) \prec_{\beta} (P_{\alpha}(y) \cdot z)\\
=&\ p(x) \cdot P_{\alpha}\big(y \cdot P_{\beta}(z)+ \lambda_{\beta} y \cdot z \big)+ \lambda_{\alpha} p(x) \cdot \big(y \cdot P_{\beta}(z)+ \lambda_{\beta} y \cdot z \big)+ p(x) \cdot P_{\beta} \big(P_{\alpha}(y) \cdot z \big)\\
&\ + \lambda_{\beta} p(x) \cdot \big(P_{\alpha}(y) \cdot z \big)\\
=&\ p(x)\big(P_{\alpha}(y) \cdot P_{\beta}(z) \big) + \lambda_{\alpha} p(x) \cdot \big(y \cdot P_{\beta}(z)\big)+ \lambda_{\alpha}\lambda_{\beta} p(x) \cdot (y \cdot z)+ \lambda_{\beta} p(x) \cdot \big(P_{\alpha}(y) \cdot z \big)\\
=&\ \big(x \cdot P_{\alpha}(y)+ \lambda_{\alpha} x \cdot y \big) \cdot P_{\beta}(p(z))+ \lambda_{\beta}(x \cdot P_{\alpha}(y)+ \lambda_{\alpha} x \cdot y) \cdot p(z)\\
=&\ (x \cdot P_{\alpha}(y) + \lambda_{\alpha} x \cdot y) \prec_{\beta} p(z)= (x \prec_{\alpha} y) \prec_{\beta} p(z).
\end{align*}
Also,
\begin{align*}
(x \succ_{\alpha} y) \prec_{\beta} p(z)=&\ (P_{\alpha}(x) \cdot y) \prec_{\beta} p(z)=(P_{\alpha}(x) \cdot y) \cdot P_{\beta}(p(z))+ \lambda_{\beta} (P_{\alpha}(x) \cdot y) \cdot p(z)\\
=&\  P_{\alpha}(p(x)) \cdot (y \cdot P_{\beta}(z))+\lambda_{\beta}P_{\alpha}(p(x))\cdot (y \cdot z)=P_{\alpha}(p(x)) \cdot (y \cdot P_{\beta}(z)+ \lambda_\beta y \cdot z)\\
=&\ P_{\alpha}(p(x)) \cdot (y \prec_{\beta} z)=p(x) \succ_{\alpha}(y \prec_{\beta} z)
\end{align*}
and
\begin{align*}
&\ (x \prec_{\beta} y) \succ_{\alpha} p(z) + (x \succ_{\alpha} y) \succ_{\beta} p(z)\\
=&\ \Big(x \cdot P_{\beta}(y)+ \lambda_{\beta} x \cdot y \Big) \succ_{\alpha} p(z)+ \Big(P_{\alpha}(x) \cdot y \Big) \succ_{\beta} p(z)\\
=&\ P_{\alpha}\Big(x \cdot P_{\beta}(y)+ \lambda_{\beta} x \cdot y \Big) \cdot p(z) + P_{\beta}\Big(P_{\alpha}(x) \cdot y \Big) \cdot p(z)\\
=&\ \Big(P_{\alpha}(x \cdot P_{\beta}(y)) + P_{\beta}(P_{\alpha}(x) \cdot y)+ \lambda_{\beta} P_{\alpha}(x \cdot y)\Big) \cdot p(z)\\
=&\ \Big(P_{\alpha}(x) \cdot P_{\beta}(y)\Big) \cdot p(z)\\
=&\  P_{\alpha}(p(x)) \cdot \big(P_{\beta}(y) \cdot z \big)\\
=&\ p(x) \succ_{\alpha}(y \succ_{\beta} z).
\end{align*}
Hence $(A,\prec_{\Omega}, \succ_{\Omega},p)$ is a matching Hom-dendriform algebra.
\end{proof}

\begin{prop}
Let $(A,\cdot,P_{\Omega},p)$ be a matching Hom-associative Rota-Baxter algebra. Assume that $p \circ P_{\omega}=P_{\omega} \circ p$ for each $\omega \in \Omega$. Define the operations $\prec_{\omega}$, $\succ_{\omega}$  and $\bullet_{\omega}$ for $\omega \in \Omega$ by
\begin{align*}
 x \prec_{\omega} y := x \cdot P_{\omega}(y), \,\,  x \succ_{\omega}y = P_{\omega}(x) \cdot y \,\,  \text{and} \,\,  x \bullet_{\omega} y = \lambda_{\omega} x\cdot y, \,\, \text{for $x,y \in A$}.
\end{align*}
Then $(A, \prec_{\Omega}, \bullet_{\Omega}, \succ_{\Omega}, p)$ is a matching Hom-tridendriform algebra.
\end{prop}
\begin{proof}
For $x, y, z \in A$ and $\alpha, \beta \in \Omega$, we have
\begin{align*}
(x \prec_{\alpha}y) \prec_{\beta}(p(z))=&\ \big(x \cdot P_{\alpha}(y)\big) \cdot P_{\beta}(p(z))=p(x) \cdot \big(P_{\alpha}(y) \cdot P_{\beta}(z) \big)\\
=&\ p(x) \cdot \big(P_{\alpha}(y \cdot P_{\beta}(z))+P_{\beta}(P_{\alpha}(y) \cdot z)+ \lambda_{\beta}P_{\alpha}(y \cdot z)  \big)\\
=&\ p(x) \prec_{\alpha}(y \prec_{\beta}(z))+p(x)\prec_{\beta}(y \succ_{\alpha} z)+x \prec_{\alpha}(y \bullet_{\beta} z),\\
(x \succ_{\alpha} y) \prec_{\beta} p(z)=&\ (P_{\alpha}(x) \cdot y) \cdot P_{\beta}(p(z))= P_{\alpha}(p(x)) \cdot (y \cdot P_{\beta}(z))= p(x) \succ_{\alpha}(y \prec_{\beta} z),\\
p(x) \succ_{\alpha}(y \succ_{\beta} z)=&\ P_{\alpha}(p(x)) \cdot \big(P_{\beta}(y) \cdot z \big)=\big(P_{\alpha}(x) \cdot P_{\beta}(y) \big) \cdot p(z)\\
=&\ \big(P_{\alpha}(x \cdot P_{\beta}(y))+P_{\beta}(P_{\alpha}(x) \cdot y)+ \lambda_{\beta} P_{\alpha}(x \cdot y)\big) \cdot p(z)\\
=&\ (x \prec_{\beta} y) \succ_{\alpha} p(z)+ (x \succ_{\alpha} y) \succ_{\beta} p(z)+ (x \bullet_{\beta} y) \succ_{\alpha} p(z),\\
(x \succ_{\alpha} y) \bullet_{\beta} p(z)=&\ \lambda_{\beta}(P_{\alpha}(x) \cdot y) \cdot p(z)= \lambda_{\beta} P_{\alpha}(p(x)) \cdot(y \cdot z)=p(x) \succ_{\alpha}(y \bullet_{\beta} z),\\
(x \prec_{\alpha} y)\bullet_{\beta} p(z)=&\ \lambda_{\beta}(x \cdot P_{\alpha}(y))\cdot p(z)=\lambda_{\beta} p(x) \cdot (P_{\alpha}(y) \cdot z)=  p(x)\bullet_{\beta} (y \succ_{\alpha} z),\\
(x \bullet_{\alpha} y) \prec_{\beta} p(z)=& \lambda_{\alpha} (x \cdot y)\cdot  P_{\beta} (p(z))=\lambda_{\alpha} p(x) \cdot (y \cdot P_{\beta}(z))= p(x) \bullet_{\alpha}(y \prec_{\beta} z),\\
(x \bullet_{\alpha} y) \bullet_{\beta} p(z)=&\ \lambda_{\alpha} \lambda_{\beta}(x \cdot y) \cdot p(z)=\lambda_{\alpha} \lambda_{\beta} p(x) \cdot(y \cdot z)=p(x) \bullet_{\alpha}(y \bullet_{\beta} z),
\end{align*}
as required.
\end{proof}

\begin{coro}
\begin{enumerate}
\item \mlabel{Item15} Let $(A,\cdot,P_{\Omega}, p)$ be a matching Hom-associative Rota-Baxter algebra of weight 0. Then $(A, \ast_{\Omega} )$ is a matching Hom-preLie algebra, where
    \begin{align*}
    x \ast_{\omega} y:=P_{\omega}(x) \cdot y-y \cdot P_{\omega}(x) \, \text{ for } x,y \in A \text{ and } \omega \in \Omega.
    \end{align*}
\item \mlabel{Item25} Let $(A,\cdot,P_{\Omega}, p)$ be a matching Hom-associative Rota-Baxter algebra. Then $(A, \ast_{\Omega})$ is a matching Hom-preLie algebra, where
    \begin{align*}
    x \ast_{\omega} y:=P_{\omega}(x) \cdot y-y \cdot P_{\omega}(x)-\lambda_{\omega}y\cdot x \,\, \text{ for $x,y \in A$ and $\omega \in \Omega$.}
    \end{align*}
\end{enumerate}
\end{coro}

\begin{proof}
(\mref{Item15}) It follows from Theorem~\mref{thm:dentopre} and Proposition~\mref{prop:rbtoden}~(\mref{Item14}).

(\mref{Item25})  It follows from Theorem~\mref{thm:dentopre} and Proposition~\mref{prop:rbtoden}~(\mref{Item24}).
\end{proof}
\section{Matching Rota-Baxter operators and Hom-nonassociative algebras} \mlabel{sec4}
Rota-Baxter Lie algebras were introduced independently by Belavin and Drinfeld and Semenov-Tian-Shansky in~\mcite{BD82, S83} and were related to solutions of the (modified) Yang-Baxter equation. Makhlouf extended Rota-Baxter operators to the context of Hom-Lie algebras. Now we generalize it to the matching Rota-Baxter case.

\begin{defn}
Let $\lambda_{\Omega}:=(\lambda_{\omega})_{\omega \in \Omega}  \subseteq \bf k$ be a family indexed by $\Omega$. A matching Hom-Lie Rota-Baxter algebra is a Hom-Lie algebra $(\mathfrak{g}, [,],p)$ endowed with a set of linear maps $P_{\omega} : \mathfrak{g} \rightarrow \mathfrak{g}$, where $\omega \in \Omega$, subject to the relation
\begin{align}
 [P_{\alpha}(x),P_{\beta}(y)] = P_{\alpha}([x, P_{\beta}(y)]) + P_{\beta}([P_{\alpha}(x), y])+ \lambda_{\beta} P_{\alpha}([x,y]),
\mlabel{eq:lrbeq}
\end{align}
for all $x,y \in \mathfrak{g}$ and $\alpha, \beta \in \Omega$. For simplicity, we denote it by $(\mathfrak{g},[,], P_{\Omega},p)$.
\end{defn}

\begin{theorem}
 Let $(\mathfrak{g}, [,],P_{\Omega})$ be a matching Lie Rota-Baxter algebra and $p: \mathfrak{g} \rightarrow \mathfrak{g}$ be a Lie algebra endomorphism such that $p \circ P_{\omega}=P_{\omega} \circ p$ for each $\omega \in \Omega$. Then $(\mathfrak{g}, [,]_{p}, P_{\Omega},p)$, where $[,]_p := p\circ [ , ]$, is a matching Hom-Lie Rota-Baxter algebra.
\end{theorem}

\begin{proof}
Since $[p(x), [y, z]_p]_p = p[p(x), p[y, z]] = p^2[x, [y, z]]$, the Hom-Jacobi identity for $(\mathfrak{g}, [,]_p, p)$ follows from the Jacobi identity of $(\mathfrak{g}, [,])$. The skew-symmetry of $(\mathfrak{g}, [,]_p, p)$ holds from the skew-symmetry of $(\mathfrak{g}, [,])$, hence $(\mathfrak{g}, [,]_p, p)$ is a Hom-Lie algebra.

For $x,y \in \mathfrak{g}$ and $\alpha, \beta \in \Omega$, we have
\begin{align*}
[P_{\alpha}(x), P_{\beta}(y)]_{p}=&\ p[P_{\alpha}(x), P_{\beta}(y)]\\
=&\  p\big(P_{\alpha}([x, P_{\beta}(y)])+ P_{\beta}[P_{\alpha}(x),y]+ \lambda_{\beta} P_{\alpha}([x,y])\big)\\
=&\ P_{\alpha}\big(p[x, P_{\beta}(y)]\big)+ P_{\beta}\big(p[P_{\alpha}(x),y]\big) +\lambda_{\beta} P_{\alpha} \big(p[x,y]\big)\\
=&\ P_{\alpha}([x,P_{\beta}(y)]_p)+ P_{\beta}([P_{\alpha}(x),y]_p)+\lambda_{\beta} P_{\alpha}([x,y]_p),
\end{align*}
as required.
\end{proof}

\begin{prop}
Let $(\mathfrak{g}, [,], P_{\Omega}, p)$ be a matching Hom-Lie Rota-Baxter algebra such that $p \circ P_{\omega}=P_{\omega} \circ p$ for each $\omega \in \Omega$. Then $(g, [,]_{p^{-1}}:= p^{-1}\circ [,],P_{\Omega})$ is a matching Lie Rota-Baxter algebra.
\end{prop}

\begin{proof}
Since $[x,[y,z]_{p^{-1}}]_{p^{-1}}=p^{-1}[x,p^{-1}[y,z]]$, the Jacobi identity of $(g, [,]_{p^{-1}}$ holds from the Hom-Jacobi identity of $(\mathfrak{g}, [,], p)$. The skew-symmetry of $(g, [,]_{p^{-1}}$ holds from skew symmetry of $(\mathfrak{g}, [,], p)$, hence $(g, [,]_{p^{-1}}$ is a Lie algebra.

Since $p \circ P_{\omega}=P_{\omega} \circ p$, $p^{-1} \circ P_{\omega}=P_{\omega} \circ p^{-1}$. Then
\begin{align*}
[P_{\alpha}(x), P_{\beta}(y)]_{p^{-1}}=&\ p^{-1}([P_{\alpha}(x), P_{\beta}(y)])\\
=&\ p^{-1} \big(P_{\alpha}([x, P_{\beta}(y)]) + P_{\beta}([P_{\alpha}(x), y])+ \lambda_{\beta} P_{\alpha}([x,y])\big)\\
=&\ P_{\alpha} \big( p^{-1}([x,P_{\beta}(y)])\big)+P_{\beta} \big(p^{-1}([P_{\alpha}(x),y]) \big)+ \lambda_{\beta} P_{\alpha} \big({p^{-1}}([x,y]) \big)\\
=&\ P_{\alpha}([x, P_{\beta}(y)]_{p^{-1}}) + P_{\beta}([P_{\alpha}(x), y]_{p^{-1}})+ \lambda_{\beta} P_{\alpha}([x,y]_{p^{-1}}),
\end{align*}
 as required.
\end{proof}

\begin{defn}
Let $(\mathfrak{g}, [,], p)$ be a multiplicative Hom-Lie algebra and $n \geq 0$. The $n$th derived Hom-algebra of $\mathfrak{g}$ is defined by
\begin{align*}
\mathfrak{g}_{(n)} = (\mathfrak{g}, [,]^{(n)} = p^n \circ [,], p^{n+1}).
\end{align*}
\end{defn}

\begin{theorem}
Let $(\mathfrak{g}, [,],P_{\Omega},p)$ be a multiplicative matching Hom-Lie Rota-Baxter algebra and assume that $p \circ P_{\omega}=P_{\omega} \circ p$ for each $\omega \in \Omega$. Then its $n$th derived Hom-algebra is a matching Hom-Lie Rota-Baxter algebra.
\end{theorem}

\begin{proof}
Following~\mcite{Y10}, the $n$-th derived Hom-algebra is a Hom-Lie algebra. For $x,y \in \mathfrak{g}$ and $\alpha, \beta \in \Omega$,
\begin{align*}
[P_{\alpha}(x), P_{\beta}(y)]^{(n)}=&\ p^n([P_{\alpha}(x),P_{\beta}(y)])\\
=&\ p^{n} \big(P_{\alpha}([x, P_{\beta}(y)]) + P_{\beta}([P_{\alpha}(x), y])+ \lambda_{\beta} P_{\alpha}([x,y])\big)\\
=&\ P_{\alpha} \big( p^{n}([x,P_{\beta}(y)])\big)+P_{\beta} \big(p^{n}([P_{\alpha}(x),y]) \big)+ \lambda_{\beta} P_{\alpha} \big(p^n([x,y]) \big)\\
=&\ P_{\alpha}([x, P_{\beta}(y)]^{(n)}) + P_{\beta}([P_{\alpha}(x), y]^{(n)})+ \lambda_{\beta} P_{\alpha}([x,y]^{(n)}),
\end{align*}
as required.
\end{proof}

In the following we construct matching Hom-Lie  Rota-Baxter algebras involving elements of the centroid of matching Lie Rota-Baxter algebras. Let $(\mathfrak{g}, [, ],\Omega, R)$ be a matching Lie Rota-Baxter algebra.
The centroid is defined by
\begin{align*}
Cent(\mathfrak{g}) := \{p \in End(\mathfrak{g}) : p[x, y] = [p(x), y], \forall x, y \in \mathfrak{g}\}.
\end{align*}

\begin{prop}
Let $(\mathfrak{g}, [,],P_{\Omega})$ be a matching Lie Rota-Baxter algebra. Let $p \in Cent(\mathfrak{g})$ and set for $x, y \in \mathfrak{g}$
\begin{align*}
[x, y]_p^1:= [p(x), y] \quad \text{and} \quad  [x, y]_p^2:= [p(x), p(y)].
\end{align*}
Assume that $p \circ P_{\omega}=P_{\omega} \circ p$ for each $\omega \in \Omega$. Then $(\mathfrak{g}, [, ]_p^1,P_{\Omega}, p)$ and $(\mathfrak{g}, [,]_p^2,P_{\Omega},p)$ are matching Hom-Lie Rota-Baxter algebras.
\end{prop}

\begin{proof}
Following Proposition~1.12 of \mcite{BM14}, $(\mathfrak{g}, [,]_p^1,p)$ and $(\mathfrak{g}, [,]_p^2, p)$ are Hom-Lie algebras. Also,
\begin{align*}
[P_{\alpha}(x), P_{\beta}(y)]_p^1 &\ = [p(P_{\alpha}(x)), P_{\beta}(y)]=p([P_{\alpha}(x), P_{\beta}(y)])\\
&\ =  p \big( P_{\alpha}([x, P_{\beta}(y)]) + P_{\beta}([P_{\alpha}(x), y])+ \lambda_{\beta} P_{\alpha}([x,y]) \big)\\
&\ = P_{\alpha}([p(x), P_{\beta}(y)]) + P_{\beta}([p(P_{\alpha}(x)),y])+ \lambda_{\beta} P_{\alpha}([p(x),y]) \\
&\ =P_{\alpha}([x, P_{\beta}(y)]^1_p)+ P_{\beta}([P_{\alpha}(x),y]^1_p)+ \lambda_{\beta}P_{\alpha}([x,y]^1_p)
\end{align*}
and
\begin{align*}
[P_{\alpha}(x), P_{\beta}(y)]_p^2 &\ = [p(P_{\alpha}(x)), p(P_{\beta}(y))]=p([P_{\alpha}(x), p(P_{\beta}(y))])\\
&\ =-p^2([P_{\beta}(y), P_{\alpha}(x)])=p^2([P_{\alpha}(x), P_{\beta}(y)])\\
&\ =  p^2 \big( P_{\alpha}([x, P_{\beta}(y)]) + P_{\beta}([P_{\alpha}(x), y])+ \lambda_{\beta} P_{\alpha}([x,y]) \big)\\
&\ = P_{\alpha}([p(x), p(P_{\beta}(y))]) + P_{\beta}([p(P_{\alpha}(x)),p(y)])+ \lambda_{\beta} P_{\alpha}([p(x),p(y)]) \\
&\ =P_{\alpha}([x, P_{\beta}(y)]^2_p)+ P_{\beta}([P_{\alpha}(x),y]^2_p)+ \lambda_{\beta}P_{\alpha}([x,y]^2_p).
\end{align*}
This completes the proof.
\end{proof}

\begin{prop}
Let $(A, [, ], P_{\Omega},p)$ be a matching Hom-Lie Rota-Baxter algebra of weight zero (i.e. $\lambda_{\omega}=0$ for all $\omega \in \Omega$). Assume that $p \circ P_{\omega}=P_{\omega} \circ p$ for each $\omega \in \Omega$.
Then $(A, \{\ast_{\omega} \,|\, \omega \in \Omega \}, p)$ is a matching Hom-pre-Lie algebra, where
\begin{align}
 x \ast_{\omega} y = [P_{\omega}(x), y] \, \text{ for }\, x,y \in A \text{ and } \omega \in \Omega.
 \mlabel{eq:mpre1}
\end{align}
\end{prop}

\begin{proof}
For $x,y,z \in \mathfrak{g}$ and $\alpha, \beta \in \Omega$, we have
\begin{align*}
&\ p(x) \ast_{\alpha} (y \ast_{\beta } z)-(x \ast_{\alpha} y) \ast_{\beta} z\\
=&\ [P_{\alpha}(p(x)), [P_{\beta}(y), z]]-[P_{\beta}([P_{\alpha}(x),y]),p(z)]\quad \text{(by Eq.~(\ref{eq:mpre1}))}\\
=&\ [P_{\alpha}(p(x)), [P_{\beta}(y),z]]-[[P_{\alpha}(x),P_{\beta}(y)],p(z)]+[P_{\alpha}([x,P_{\beta}(y)]),p(z)]\quad \text{(by Eq.~(\ref{eq:lrbeq}))}\\
=&\ [p(P_{\alpha}(x)), [P_{\beta}(y),z]]+[p(z), [P_{\alpha}(x), P_{\beta}(y)]]-[P_{\alpha}([P_{\beta}(y), x]), p(z)]\quad \text{(by $p \circ P_{\alpha}=P_{\alpha} \circ p$)}\\
=&\ -[p(P_{\beta}(y)),[z,P_{\alpha}(x)]]-[P_{\alpha}([P_{\beta}(y), x]), p(z)]\quad \text{(by Hom-Jacobi identity)}\\
=&\ [P_{\beta}(p(y)), [P_{\alpha}(x),z]]-[P_{\alpha}([P_{\beta}(y), x]), p(z)]\\
=&\ p(y) \ast_{\beta} (x \ast_{\alpha} z)-(y \ast_{\beta} x) \ast_{\alpha} p(z).
\end{align*}
This completes the proof.
\end{proof}
\medskip
\noindent {\bf Data Availability:}
No data were used to support this study.

\smallskip

\noindent {\bf Disclosure statement:}
No potential conflict of interest was reported by the authors.
\medskip

\noindent {\bf Acknowledgments}: The authors would like to thank the anonymous reviewers for their valuable comments which improved the paper. This work was supported by the National Natural Science Foundation of China (Grant No.\@ 11771191). The work of Chia Zargeh was partially supporetd by CNPq grant.

\end{document}